\documentclass[10pt]{amsart}
\usepackage{amssymb,amsmath,txfonts,amsthm}
\usepackage{hyperref}
\usepackage{mathrsfs}
\newtheorem{theorem}{Theorem}
\newtheorem{prop}{Proposition}
\newtheorem{lemma}{Lemma}
\newtheorem*{remark}{Remark}
\newtheorem{claim}{Claim}
\newtheorem{definition}{Definition}
\newtheorem{cor}{Corollary}
\newtheorem{conj}{Conjecture}

\def\XXint#1#2#3{{\setbox0=\hbox{$#1{#2#3}{\int}$}
  \vcenter{\hbox{$#2#3$}}\kern-.5\wd0}}

\author{Gang Liu}
\address{Department of Mathematics\\University of California, Berkeley\\Berkeley, CA 94720}
\email{gangliu@math.berkeley.edu}
\title[Three Circle Theorem]{Three Circle Theorems on K\"ahler manifolds and applications}
\date{}
\begin{document}
\begin{abstract}The classical Hadamard three circle theorem is generalized to complete K\"ahler manifolds. More precisely, we show that the nonnegativity of the holomorphic sectional curvature is a necessary and sufficient condition for the three circle theorem.  As corollaries, two sharp monotonicity formulae for holomorphic functions are derived. Among applications, we derive sharp dimension estimates (with rigidity) of holomorphic functions with polynomial growth when the holomorphic sectional curvature is nonnegative. When the bisectional curvature is nonnegative, this was due to Ni.  Also we study holomorphic functions with polynomial growth near infinity. On a complete noncompact K\"ahler manifold with nonnegative bisectional curvature, we prove any holomorphic function with polynomial growth is homogenous at infinity. This result is closely related with Yau's conjecture on the finite generation of holomorphic functions (see page $3-4$ for detailed explanation).

We also generalize the three circle theorem to K\"ahler manifolds with holomorphic sectional curvature lower bound. These inequalities are sharp in general. As applications, we establish sharp dimension estimates for holomorphic functions with certain growth conditions when the holomorphic sectional curvature is asymptotically nonnegative.

Finally observe that Hitchin constructed some complex manifolds which admit K\"ahler metric with positive holomorphic sectional curvature, yet do not even admit K\"ahler metrics with nonnegative Ricci curvature.  Thus the set of complete K\"ahler manifolds with nonnegative holomorphic sectional curvature is strictly wider than those with nonnegative bisectional curvature.
\end{abstract}
\maketitle

\section{\bf{Introduction}}
Recall the
 classical Hadamard Three Circle Theorem:
\begin{theorem}
Let $f(z)$ be a holomorphic function on the annulus $r_1\leq |z|\leq r_3$. Let $M_f(r)$ be the maximum of $|f(z)|$ on the circle $|z|=r$. Then $\log M_f(r)$ is a convex function of $\log r$.  In other words,  
\begin{equation}
\log(\frac{r_3}{r_1})\log M_f(r_2) \leq \log(\frac{r_3}{r_2})\log M_f(r_1) + \log(\frac{r_2}{r_1})\log M_f(r_3)
\end{equation}
where $r_1 < r_2 < r_3$.
\end{theorem}

It is natural to wonder whether we can generalize this theorem to complete K\"ahler manifolds.
 \begin{definition}
 Let $M$ be a complete K\"ahler manifold. We say $M$ satisfies the three circle theorem, if for any point $p\in M$, $r>0$, any holomorphic function $f$ on $B(p, r)$, $\log M_f(r)$ is a convex function of $\log r$. In other words,  for $r_1 < r_2 < r_3$,
\begin{equation}\label{eq-200}
\log(\frac{r_3}{r_1})\log M_f(r_2) \leq \log(\frac{r_3}{r_2})\log M_f(r_1) + \log(\frac{r_2}{r_1})\log M_f(r_3).
\end{equation}
Here $M_f(r) = \max |f(x)|$ for $x\in B(p, r)$.
\end{definition}

We have a very satisfactory classification of manifolds satisfying the three circle theorem.
\begin{theorem}\label{thm-10}
Let $M$ be a complete K\"ahler manifold. Then $M$ satisfies the three circle theorem if and only if the holomorphic sectional curvature is nonnegative.
\end{theorem}
\begin{remark}\label{rm1}
It is interesting to compare this theorem with the classical volume comparison theorem. The volume comparison theorem says the Ricci curvature controls the volume growth; theorem \ref{thm-10} says the holomorphic sectional curvature controls the growth of holomorphic functions. 
\end{remark}
\begin{remark}
There are generalizations of the three circle theorem in various settings. For instance, Colding, De Lellis and Minicozzi \cite{[CLM]} proved a three circle inequality for the Sobolev norm of the solution to a Schr\"odinger equation on a product $N\times [0, T]$ for certain closed manifold $N$.  \end{remark}
\begin{remark}
In the Riemannian case, Colding-Minicozzi \cite{[CM1]} introduced a frequency function for harmonic functions.
In remark $2.16$ of \cite{[CM1]}, Colding and Minicozzi pointed out that when the manifold is Euclidean, the monotonicity of the frequency is an analytic version of the Three Circle Theorem of Hadamard. However, the frequency function is not monotonic for general manifolds with nonnegative Ricci curvature and Euclidean volume growth. Moreover, the order of a harmonic function at infinity is not well defined (see 
page 67 in \cite{[CM1]}). Compare this with corollary \ref{cor0} and corollary \ref{cor1}. 
\end{remark}
Next we discuss some applications of theorem \ref{thm-10}. Recall three conjectures of Yau \cite{[Y1]}.
\begin{conj}\label{conj1}
Let $M^n$ be a complete noncompact K\"ahler manifold with positive bisectional curvature, then $M$ is biholomorphic to $\mathbb{C}^n$.
\end{conj}
On a complete K\"ahler manifold $M$,  we say a holomorphic function $f\in \mathcal{O}_d(M)$, if there exists some $C>0$ with $|f(x)|\leq C(1+r(x, x_0))^d$ for all $x\in M$. Here $x_0$ is a fixed point on $M$. See also definition \ref{def1}.
\begin{conj}\label{conj2}
Let $M^n$ be a complete noncompact K\"ahler manifold with nonnegative bisectional curvature. Then the ring of holomorphic functions with polynomial growth is finitely generated.\end{conj}
\begin{conj}\label{conj3}
Let $M^n$ be a complete noncompact K\"ahler manifold with nonnegative bisectional curvature. Then given any $d>0$, $dim(\mathcal{O}_d(M))\leq dim(\mathcal{O}_d(\mathbb{C}^n))$. \end{conj}

So far there have been many works on conjecture \ref{conj1}. For example,  \cite{[Mo]}\cite{[MSY]}\cite{[CT]}, etc. However, conjecture \ref{conj1} is still open in general.
Below we shall consider some aspects of conjecture \ref{conj3} and conjecture \ref{conj2}.

To bound of the dimension of holomorphic functions with polynomial growth, it suffices to bound the vanishing order by the growth rate.

In \cite{[Mo]}, Mok proved the following:
\begin{theorem}[Mok] 
Let $M^n$ be a complete K\"ahler manifold with positive Ricci curvature ($n\geq 2$), such that for some fixed point $p\in M$, 
\begin{itemize}
\item Scalar curvature $\leq \frac{C_0}{r(p, x)^2}$ for some $C_0>0$;
\item $Vol(B(p, r)\geq C_1r^{2n}$ for some $C_1>0$.
\end{itemize}
Let $f\in \mathcal{O}_d(M)$, then there exists a constant $C$ independent of $f$ such that the vanishing order of $f$ at $p$ satisfies $ord_p(f) \leq Cd$. \end{theorem}

For a complete K\"ahler manifold $M^n$ with nonnegative bisectional curvature, we say $M$ is of maximal volume growth, if $\lim\limits_{r\to\infty}\frac{Vol(B(p, r))}{r^{2n}} > 0$.
In \cite{[N1]}, Ni proved a remarkable theorem which confirmed conjecture \ref{conj3} by assuming the maximal volume growth condition:
\begin{theorem}[Ni]\label{thm0}
Let $M^n$ be a complete K\"ahler manifold with nonnegative holomorphic bisectional curvature. Assume $M$ is of maximal volume growth, then $$dim(\mathcal{O}_d(M))\leq dim(\mathcal{O}_d(\mathbb{C}^n))$$ for any positive integer $d$. If the equality holds for some $d$, $M$ is isometric and biholomorphic to $\mathbb{C}^n$.
\end{theorem}
Ni's method is parabolic.  There are two key points in the proof: 1. The plurisubharmonicity is preserved under the heat flow when the bisectional curvature is nonnegative. This results was proved in \cite{[NT]} by Ni and Tam. 2. A monotonicity formula for plurisubharmonic functions under the heat flow. This monotonicity formula resembles the trace form of Hamilton's Li-Yau-Hamilton differential inequality \cite{[H1]} originally called the differential Harnack inequality, for the Ricci flow. 

 By using same technique in \cite{[N1]},  Chen, Fu, Le and Zhu \cite{[CFLZ]} removed the maximal volume growth condition in theorem \ref{thm0}. Therefore, conjecture \ref{conj3} was solved completely.

\bigskip
As an application of theorem \ref{thm-10},
 we generalize theorem \ref{thm0} to the case when the holomorphic sectional curvature is nonnegative. 
 \begin{theorem}\label{thm2}
 
 Let $M^n$ be a complete noncompact K\"ahler manifold with nonnegative holomorphic sectional curvature. Then for any $d > 0$, $dim(\mathcal{O}_d(M)) \leq dim(\mathcal{O}_d(\mathbb{C}^n))$. Moreover, if the equality holds for some positive integer $d$, $M$ is biholomorphic and isometric to $\mathbb{C}^n$.
 \end{theorem}
The condition in theorem \ref{thm2} is much weaker than the nonnegativity of holomorphic bisectional curvature.  For instance, it is not known whether the volume comparison theorem holds in this case.

In the Riemannian case, there have been many articles on the dimension estimate of harmonic
functions with polynomial growth. See \cite{[CM]}-\cite{[CM6]}, \cite{[Li4]} etc.  Colding and Minicozzi \cite{[CM]} proved that for a complete manifold $M^m$ with nonnegative Ricci curvature, the dimension of harmonic functions with polynomial growth is finite. In \cite{[Li4]}, Li provided an elegant short proof. However, the example of Donnelly \cite{[D]} shows that the sharp estimate is not true, if we compare with the Euclidean space. The sharp upper bound was only obtained either when $d = 1$ or $m = 2$ in \cite{[LT1]}\cite{[LT2]} by Li and Tam. The rigidity part for $d = 1$ was due to Li \cite{[Li2]} for the K\"ahler case; Cheeger-Colding-Minicozzi \cite{[CCM]} for the general case.  Li and Wang \cite{[LW1]} showed that when the sectional curvature is nonnegative and manifold has maximal volume growth,  an asymptotically sharp estimate is valid. Thus for dimension estimates of harmonic (holomorphic) functions, there is a subtle difference between the Riemannian case and  the K\"ahler case. Please refer to \cite{[Li1]}\cite{[Li3]} for nice survey of these results.

\medskip

Now we consider conjecture \ref{conj2}.
So far very little is known. In \cite{[LT2]}, Li and Tam confirmed conjecture \ref{conj2} in complex dimension $1$. Due to the specialty in this case, it seems not so easy to generalize their result to higher dimensions.  In \cite{[CM6]}\cite{[N1]}, it was proved that quotient field of holomorphic functions with polynomial growth has finite transcendental degree over $\mathbb{C}$. However, it is unclear why the ring itself is finitely generated.

The difficulty of conjecture \ref{conj2} is to build a bridge between analysis and algebra. We propose a method to solve conjecture \ref{conj2}. More precisely, 
in theorem \ref{thm-99}, we prove that on a complete K\"ahler manifold with nonnegative bisectional curvature, any holomorphic function with polynomial growth is asymptotically homogenous. The point here is that these annuli are independent of the holomorphic functions and no volume condition is assumed( in \cite{[L]}, it was proved that if the manifold does not split, it must be of maximal volume growth). Therefore, we can handle all holomorphic functions with polynomial growth in one single tangent cone.   Moreover, if there exists a nonconstant holomorphic function with polynomial growth and the universal cover of the manifold does not split, it was proved that the manifold must be of maximal volume growth.  In some sense, we reduce the problem to the compact factor of the metric cone. Then, techniques from algebraic geometry might be used to show finite generation of the compact factor. The picture is clear for standard $\mathbb{C}^n$: given any holomorphic polynomial $f$, $f$ is homogenous with degree $d$ at infinity. Then $f$ can be reduced to a function on $\mathbb{S}^{2n-1}$ which is the cross section of the tangent cone of $\mathbb{C}^n$. By the Hopf fibration $\mathbb{S}^1\to \mathbb{S}^{2n-1}\to \mathbb{C}\mathbb{P}^{n-1}$, we further reduce the function to a global section of the line bundle $dH$, where $H$ is the hyperplane line bundle on $\mathbb{C}\mathbb{P}^{n-1}$.

In theorem \ref{thm8}, the three circle theorem is extended to the case when the holomorphic sectional curvature has a lower bound. As applications, we generalize theorem \ref{thm2} to the case when the holomorphic sectional curvature is only asymptotically nonnegative. The power is still sharp. This in particular includes the case when the holomorphic sectional curvature is nonnegative outside a compact set. To the author's best knowledge, even with the stronger condition that the bisectional curvature is nonnegative outside a compact set, it was not known whether the dimension of holomorphic functions with polynomial growth is finite. Also, in the Riemannian case, it was only known that if the first Betti number is finite and the manifold has nonnegative Ricci curvature outside a compact set, the dimension of harmonic functions with polynomial growth is finite. See for example, \cite{[W]}\cite{[T]}.

We also apply the three circle theorem to holomorphic maps with polynomial growth between certain complete K\"ahler manifolds. We are able to embed the space of holomorphic maps into a finite dimensional complex manifold.  If the embedding is a homeomorphism, we show both the domain and the target are complex Euclidean spaces.

The proof of theorem \ref{thm-10} is surprisingly simple. It only uses the Hessian comparison theorem and the maximum principle. In the maximum principle, we consider the Hessian in one direction, not the Laplacian.

Apart from the introduction, this paper is organized as follows:
In section $2$, we prove theorem \ref{thm-10}. Two monotonicity formulae are derived as corollaries. We prove theorem \ref{thm2} as a byproduct.  In section $3$, we study the asymptotic behavior of holomorphic functions with polynomial growth on K\"ahler manifolds with nonnegative bisectional curvature.
In section $4$, a Liouville type theorem for plurisubharmonic functions is proved when the holomorphic sectional curvature is nonnegative. Note that when the holomorphic bisectional curvature is nonnegative, the result was due to Ni and Tam \cite{[NT]}.
Section $5$ deals with the holomorphic bundle case. In section $6$, we study the space of holomorphic maps between certain noncompact K\"ahler manifolds. Section $7$ introduces the three circle theorem in the general setting. We assume that the holomorphic sectional curvature has a lower bound which might depend on the distance. In particular, we state the corresponding theorem when the holomorphic sectional curvature is no less than $-1$ or $1$.
In section $8$, we study holomorphic functions on complete K\"ahler manifolds with holomorphic sectional curvature asymptotically nonnegative. Sharp dimension estimates are obtained. Miscellaneous results are proved in section $9$. For example, we prove that in some cases, holomorphic functions with exponential growth have finite dimension. Moreover, the power is sharp. Section $10$ concludes this paper with examples showing that certain complex manifolds admit K\"ahler metric with positive holomorphic sectional curvature, but do not admit K\"ahler metric with nonnegative Ricci curvature.  

\bigskip
\bigskip
\vskip.1in
\begin{center}
\bf  {\quad Acknowledgment}
\end{center}
The author would like to express his deep gratitude to his former advisor, Professor Jiaping Wang for numerous help and many valuable discussions during the work. He thanks his mentor Professor John Lott, for his patience in reading the preprint and useful suggestions and corrections.
He also thanks Professors Peter Li, Lei Ni, Munteanu Ovidiu, ShingTung Yau,  for their interests in this work. Special thanks also go to Guoyi Xu for the discussion on the Gromov-Hausdorff convergence of manifolds.

\section{\bf{Hadamard Three Circle Theorem, Special Case}}
In this section we prove theorem \ref{thm-10} and theorem \ref{thm2}.

\begin{definition}\label{def1}
Let $M$ be a complete noncompact K\"ahler manifold. Let $\mathcal{O}(M)$ be the ring of holomorphic functions on $M$. For any $d\geq 0$, define $$\mathcal{O}_d(M) = \{f\in \mathcal{O}(M)|\overline{\lim\limits_{r\to\infty}}\frac{M_f(r)}{r^d}<\infty\}.$$ Here $r$ is the distance from a fixed point $p$ on $M$; $M_f(r)$ is the maximal modulus of $f$ on $B(p, r)$. If $f\in \mathcal{O}_d(M)$, we say $f$ is of polynomial growth with order $d$.
 We also define $$\mathcal{O}'_d(M)= \{f\in \mathcal{O}(M)|\lim\limits_{\overline{r\to\infty}}\frac{M_f(r)}{r^d}< \infty\}.$$ Clearly $\mathcal{O}_d(M)\subseteq  \mathcal{O}'_d(M)$.
 \end{definition}
 We state some corollaries of theorem \ref{thm-10}.
\begin{cor}[Sharp Monotonicity I]\label{cor1}
Let $M$ be a complete K\"ahler manifold with nonnegative holomorphic sectional curvature, then $f\in \mathcal{O}'_d(M)$ if and only if $\frac{M_f(r)}{r^d}$ is nonincreasing.
\end{cor}
\begin{remark}
In \cite{[C10]}\cite{[CM8]}, Colding, Colding-Minicozzi derived some important monotonicity formulae for harmonic functions.
\end{remark}

\begin{proof}
If  $\frac{M_f(r)}{r^d}$ is nonincreasing, it is obvious that $f\in \mathcal{O}'_d(M)$.
Let $f\in \mathcal{O}'_d(M)$. We need to show that $$\frac{M_f(r_1)}{r_1^d} \geq \frac{M_f(r_2)}{r_2^d}$$ for $r_1 \leq r_2$.
By rescaling, we may assume $r_1 = 1$. By the assumption of corollary \ref{cor1}, given any $\epsilon > 0$, there exists a sequence $\lambda_j\to\infty$ such that, $$\log M_f(\lambda_j) \leq \log M_f(1) + (d+\epsilon)\log \lambda_j.$$  If we take $r_3 = \lambda_j$ sufficiently large in (\ref{eq-200}), then $$M_f(r_2) \leq M_f(r_1)r_2^{d+\epsilon}.$$ The corollary follows if $\epsilon\to 0$.
\end{proof}

\begin{cor}\label{cor0}
Let $M$ be a complete noncompact K\"ahler manifold with nonnegative holomorphic sectional curvature, then $\mathcal{O}'_d(M)= \mathcal{O}_d(M)$.
\end{cor}

\begin{cor}\label{cor-7}
Let $M$ be a complete K\"ahler manifold with nonnegative holomorphic sectional curvature. Let $d\geq 0$. Assume $f\in \mathcal{O}_{d+\epsilon}(M)$ for any $\epsilon>0$. Then $f\in \mathcal{O}_{d}(M)$.
\end{cor}\begin{proof}
According to corollary \ref{cor1}, $\frac{M(r)}{r^{d+\epsilon}}$ is nonincreasing for any $\epsilon>0$. Thus $\frac{M(r)}{r^d}$ is nonincreasing. This implies that $f\in \mathcal{O}_{d}(M)$.
\end{proof}

 Corollary \ref{cor1} implies the sharp dimension estimate for holomorphic functions with polynomial growth. For reader's convenience, we rewrite theorem \ref{thm2} below.
\begin{theorem}
 Let $M^n$ be a complete noncompact K\"ahler manifold with nonnegative holomorphic sectional curvature. Then for any $d > 0$, $dim(\mathcal{O}_d(M)) \leq dim(\mathcal{O}_d(\mathbb{C}^n))$. Moreover, if the equality holds for some positive integer $d$, $M$ is biholomorphic and isometric to $\mathbb{C}^n$.
 \end{theorem}
 
 \begin{proof}
 Suppose for some $d$, the reverse inequality holds. By linear algebra, for any point $p\in M$, there exists a nonzero holomorphic function $f \in \mathcal{O}_d(M)$ such that the vanishing order at $p$ is at least $[d]+1$. Here $[d]$ is the greatest integer less than or equal to $d$. Therefore $$\lim\limits_{r \to 0^+} \frac{M_f(r)}{r^d} = 0.$$ Corollary \ref{cor1} says $\frac{M_f(r)}{r^d}$ is nonincreasing. Thus $f \equiv 0$. This is a contradiction. We postpone the proof of the rigidity to the end of this section.
 \end{proof}
 \begin{remark}
 As we remarked in the introduction, one key point of Ni's method to theorem \ref{thm0} is that the plurisubharmonicity is preserved under the heat flow when the bisectional curvature is nonnegative \cite{[NT]}.  It does not seem obvious to the author whether this still holds when the holomorphic sectional curvature is nonnegative.\end{remark}

\begin{cor}
Let $M^n$ be a complete K\"ahler manifold with nonnegative holomorphic sectional curvature. Then the quotient field associated with the holomorphic functions with polynomial growth has transcendental degree at most $n$ over $\mathbb{C}$.
\end{cor}
\begin{proof}
The proof follows from the standard Poincare-Siegel type argument. Please refer to \cite{[Mo]} for details.
\end{proof}
\begin{remark}
Analogous results were proved in \cite{[N1]}\cite{[CM6]}.
\end{remark}

\begin{definition}
On a manifold $M$, two metrics $g_1, g_2$ are quasi-isometric if there exist $C_1, C_2>0$ such that for any $p, q\in M$,
$$\frac{1}{C_1}d_{g_1}(p, q)-C_2\leq d_{g_2}(p, q) \leq C_1d_{g_1}(p, q)+C_2.$$
\end{definition}

\begin{cor}
Let $g$ be a complete K\"ahler metric on $\mathbb{C}^n$ which is quasi-isometric to the Euclidean metric $g_0$. Assume $g$ has nonnegative holomorphic sectional curvature, then $g$ is flat.
\end{cor}
\begin{proof}
According to the assumption, all linear holomorphic functions on $(\mathbb{C}^n, g_0)$ are of linear growth with respect to $g$. According to the rigidity in theorem \ref{thm2}, $g$ is flat.
\end{proof} 
 
 \begin{cor}[Sharp Monotonicity II]\label{cor2}
Let $M$ be a complete K\"ahler manifold with nonnegative holomorphic sectional curvature, $f\in \mathcal{O}(M)$. If the vanishing order of $f$ at $p$ is at least $k$, $\frac{M_f(r)}{r^k}$ is nondecreasing.
\end{cor}
\begin{remark}
Recall the classical Schwarz Lemma:  $f$ is a holomorphic function from the unit disk $D^2$ to $D^2$, with $f(0) = 0$. Then $|f(z)|\leq |z|$ and $|f'(0)|\leq 1$. 
This corollary could be regarded as a generalization of the Schwarz lemma: Just endow $D^2$ with the standard Euclidean metric and apply the corollary.\end{remark}
\begin{proof}
We need to show $$\frac{M_f(r_1)}{r_2^k} \leq \frac{M_f(r_2)}{r_3^k}$$ for $r_2\leq r_3$. By rescaling, we assume $r_3 = 1$. Since the vanishing order of $f$ at $p$ is at least $k$, given any small $\epsilon >0$,
for $r$ sufficiently small, $$\log M_f(r) \leq \log M_f(1) + (k-\epsilon)\log r.$$ Take $r_1 = r$ sufficiently small in (\ref{eq-200}). We obtain $$M_f(r_2) \leq M_f(r_3)r_2^{k-\epsilon}.$$ The corollary follows if $\epsilon \to 0$.
 \end{proof} 
\begin{cor}[Weak monotonicity for mean values]\label{cor-10}
Let $M^n$ be a complete K\"ahler manifold with nonnegative holomorphic bisectional curvature and $p\in M$.   Let $f\in \mathcal{O}(M)$ and define $A_f(r) = \frac{\int_{B(p, r)}|f|^2}{Vol(B(p, r))}$.  There exists a constant $C=C(n)$ such that
\begin{itemize}
\item $f\in \mathcal{O}_d(M)$ if and only if $\frac{CA_f(r_1)}{(\frac{r_1}{2})^{2d}}\geq \frac{A_f(r_2)}{r_2^{2d}}$ for any  $0<r_1\leq r_2$.
\item $f$ vanishes at $p$ with order at least $k$ if and only if $\frac{A_f(r_1)}{2^{2k}r_1^{2k}C}\leq \frac{A_f(r_2)}{r_2^{2k}}$ for any $0< 2r_1\leq r_2$.
\end{itemize}
\end{cor} 
\begin{proof}
According to the mean value inequality by Li and Schoen \cite{[LS]}, if $u$ is a nonnegative subharmonic function, there exists a positive constant $C(n)$ depending only on dimension such that $$C(n)\frac{\int_{B(p, r)}u^2}{Vol(B(p, r))}\geq (u(p))^2.$$  Take $u = |f|$, then $u$ is subharmonic. Together with the volume comparison, we find that there exists a constant $C'$ depending only on $n$ such that 
$$(M_f(r))^2 \geq A_f(r)\geq C'(M_f(\frac{r}{2}))^2$$ for any $r>0$. Then corollary \ref{cor-10} follows from corollary \ref{cor1} and corollary \ref{cor2}.
\end{proof}
 Now we turn to the proof of theorem \ref{thm-10}.
 \begin{proof}
 First we prove that if $M$ has nonnegative holomorphic sectional curvature, $M$ satisfies the three circle theorem.
 We will use the following Hessian comparison theorem which appears in \cite{[LW2]} by Li and Wang:
 \begin{theorem}\label{thm3}
 Let $M$ be a complete K\"ahler manifold with nonnegative holomorphic sectional curvature. Let $r$ be the distance function to a point $p\in M$. Define $e_1 = \frac{1}{\sqrt{2}}(\nabla r - \sqrt{-1}J\nabla r)$. Then when $r$ is smooth, $r_{1\overline{1}} \leq \frac{1}{2r}$. Thus $$(\log r)_{1\overline{1}} =\frac{r_{1\overline{1}}}{r} - \frac{r_1r_{\overline{1}}}{r^2}\leq 0.$$
 \end{theorem}
 Since the proof is simple, we include it here for reader's convenience.
 \begin{proof}
Let $x$ be a point outside the cut locus of $p$ on $M$. Let $e_i$ be a unitary frame at $x\in M$ and $e_1 = \frac{1}{\sqrt{2}}(\nabla r - \sqrt{-1}J\nabla r)$. We parallel transport the unitary frame along the minimal geodesic from $p$ to $x$. Consider the Bochner formula
 \begin{equation}
\begin{aligned}
 0 &= \frac{1}{2}|\nabla r|^2_{1\overline 1} \\&= r_{i1\overline{1}}r_{\overline{i}} + r_{i1}r_{\overline{i}\overline{1}}+ r_{i\overline{1}}r_{\overline{i}1}+r_{i}r_{\overline{i}1\overline{1}}\\&\geq
 2r_{1\overline{1}}^2 + \frac{\partial r_{1\overline{1}}}{\partial r} + \frac{1}{2}R_{1\overline{1}1\overline{1}}.
 \end{aligned}
 \end{equation}
 
Together with the initial condition of $r_{1\overline{1}}$ for $r\to 0$, we obtain the proof.
 \end{proof}
Define $$F(x) = \log(\frac{r_3}{r})\log M_f(r_1) + \log(\frac{r}{r_1})\log M_f(r_3)$$ for $0<r_1\leq r=dist(x, p) \leq r_3$. Define $$G(x) = \log(\frac{r_3}{r_1})\log |f(x)|.$$ It is clear that $F(x)\geq G(x)$ when $x\in \partial B(p, r_1)$  and $x\in \partial B(p, r_3)$. Suppose $F(x) < G(x)$ somewhere inside the annulus. Let $q$ be the maximum point of $G(x)-F(x)$.
If $q$ is not on the cut locus of $p$, $$\sqrt{-1}\partial\overline\partial (G(x) - F(x))|_{x=q} \leq 0;$$ $$\nabla F(q) = \nabla G(q)= C\nabla r(q)$$ where $C$ is some constant (this point will be used in the proof of theorem \ref{thm4}). In particular, $$G_{1\overline{1}} - F_{1\overline{1}}\leq 0$$ at $q$, where $e_1 = \frac{1}{\sqrt{2}}(\nabla r - \sqrt{-1}J\nabla r)$.  The Poincare-Lelong equation says $$\frac{\sqrt{-1}}{2\pi}\partial\overline\partial\log |f|^2 = [D]$$ where $D$ is the divisor of $f$. Since $G(x)-F(x)$ has the maximum at $q$, $|f(q)| \neq 0$. Therefore $$G_{1\overline{1}} = 0.$$ As $f$ is a holomorphic on $M$, we may assume $M_f(r_1) < M_f(r_3)$, otherwise $f$ is a constant.  By theorem \ref{thm3}, $$F_{1\overline{1}} \leq 0.$$ Therefore $$G_{1\overline{1}} - F_{1\overline{1}}\geq 0.$$ If the strict inequality holds, this will be a contradiction. For a small constant $\epsilon>0$, we define $$F_\epsilon(x) =a_\epsilon\log(r-\epsilon)+b_\epsilon.$$ Here $a_\epsilon$ and $b_\epsilon$ are chosen so that $$F_\epsilon(r_1)= M_f(r_1); F_\epsilon(r_3) = M_f(r_3).$$  Clearly when $\epsilon\to 0$, $F_\epsilon \to F$. Since $M_f(r_3) > M_f(r_1)$, $a_\epsilon >0$.
 Now we compute $$(\log (r-\epsilon))_{1\overline{1}} = \frac{r_{1\overline{1}}}{r-\epsilon}-\frac{1}{2(r-\epsilon)^2} < 0$$ by theorem \ref{thm3}. Therefore $$(G(x)-F_\epsilon(x))_{1\overline{1}}>0.$$ This implies that it cannot assume the maximum inside the annulus if we ignore the cut locus problem. Considering the boundary data, we find that for any small $\epsilon > 0$, $$
 G(x)-F_{\epsilon}\leq 0$$ inside the annulus.
 
 To complete the proof, we still need to handle the case when $q_\epsilon$ ($q_\epsilon$ is the maximum point of $G-F_\epsilon$) lies on the cut locus of $p$. We will adopt the trick of Calabi. Pick a number $\epsilon_1$ with $0< \epsilon_1 < \epsilon$.  Let $p_1$ be the point on the minimal geodesic connecting $p$ and $q_\epsilon$ such that $dist(p, p_1) = \epsilon_1$.  Define $$\hat{r}(x)= dist(p_1, x)$$ and consider $$F_{\epsilon, \epsilon_1} = a_\epsilon\log(\hat{r}+\epsilon_1-\epsilon)+b_\epsilon.$$ Then $$F_{\epsilon,  \epsilon_1}\geq F_\epsilon$$ by triangle inequality and $$F_{\epsilon}(q_\epsilon) =F_{\epsilon, \epsilon_1}(q_\epsilon).$$ Thus $G-F_{\epsilon, \epsilon_1}$ has maximum at $q_\epsilon$. Note $$(F_{\epsilon, \epsilon_1})_{1\overline{1} }< 0,$$ since $\epsilon_1<\epsilon$. Then we apply the maximum principle for $G-F_{\epsilon, \epsilon_1}$ at $q_\epsilon$ to get a contradiction. The proof of the sufficient part of theorem \ref{thm-10} follows if we let $\epsilon_1 \to 0$ and then $\epsilon\to 0$.
 
 We come to the necessary part of theorem \ref{thm-10}. The argument is local.  Consider the normal coordinate $U(z_1, z_2,...., z_n)$ at $p$. Let $\gamma = (z^1(t), z^2(t),...., z^n(t))$ be a geodesic emanating from $p$ and $t$ is the arc length.  Here we are assuming $z_i = \frac{1}{\sqrt{2}}(x_i+\sqrt{-1}y_i)$. This makes $|\frac{\partial}{\partial z_i}| = 1$ at $p$. Suppose $R_{1\overline 11\overline 1}<0$.
Recall the geodesic equation
\begin{equation}\label{eq100}
\frac{d^2 z^\lambda}{dt^2} = -\Gamma_{\mu\nu}^\lambda\frac{dz^\mu}{dt}\frac{dz^\nu}{dt}.
\end{equation}

Suppose the initial tangent vector of $\gamma$ is given by $X = \sum X^i\frac{\partial}{\partial z_i} +\sum \overline  X^i\frac{\partial}{\partial\overline{z_i}}$.
Note that at $p$, $R_{i\overline jk\overline l} =-\frac{\partial^2g_{i\overline j}}{\partial z_k\partial\overline{z_l}}$.
By (\ref{eq100}), \begin{equation}\label{eq101}
z^i(t) = tX^i + \frac{1}{6}\sum\limits_{j, k, l}R_{j\overline ik\overline l}(p)X^jX^k\overline{X^l}t^3+o(t^3).
\end{equation}
For the holomorphic function $z_1$, assume the maximum modulus on $B(p, r)$ is achieved at $q(r)\in \partial B(p, r)$ (note $q(r)$ might not be unique).  Let $exp$ be the exponential map and suppose $X_r$ is a unit tangent vector at $p$ with $exp_p(X_r, r) = q(r)$.  Then when $r\to 0$, $|X_r^1|\to \frac{1}{\sqrt{2}}$, \begin{equation}
\label{eq103} |X_r^i|\to 0\end{equation} for $i\neq 1$.
By (\ref{eq101}) and (\ref{eq103}),
\begin{equation}\label{eq102}
|z^1(q(t))|^2 = t^2|X_t^1|^2+\frac{1}{3}t^4\sum\limits_{j, k, l}Re(R_{j\overline 1k\overline l}(p)\overline{X_t^1}X_t^jX_t^k\overline{X_t^l})+o(t^4)\leq \frac{1}{2}t^2+\frac{1}{6}R_{1\overline 11\overline 1}|X_t^1|^4t^4
<\frac{1}{2}t^2 \end{equation} for all small $t$. In the inequality, we have used that $R_{1\overline 11\overline 1}<0$.

Let $M(r) = \max |z_1|$ on $B(p, r)$. Then $\frac{M(r)}{r} < \frac{1}{\sqrt{2}}$ for small $r>0$.  Note that $\lim\limits_{r\to 0}\frac{M(r)}{r} = \frac{1}{\sqrt{2}}$. If $M$ satisfies the three circle theorem, we have corollary \ref{cor2}. This is a contradiction. 
 
 \end{proof}

Now we complete the rigidity part of theorem \ref{thm2}.
It suffices to show that $M$ is flat (then we take the universal cover). We need to show that for any $p\in M$ and $X\in T_p(M)$,  $R_{XJXJXX}= 0$. There exists a local coordinate chart $(U, z_1, z_2,...z_n)$ containing $p$ such that $X = \frac{\partial}{\partial x_1}$ and $\frac{\partial}{\partial z_1}(p) = X - \sqrt{-1}JX$. Since $dim(\mathcal{O}_d(M)) = dim(\mathcal{O}_d(\mathbb{C}^n))$, there exists $f\in \mathcal{O}_d(M)$ such that the restriction of $f$ in $U$ is $$f(z_1, z_2,..., z_n) = z_1^d+O(r^{d+1}).$$ In particular, the vanishing order of $f$ at $p$ is $d$. Corollary \ref{cor1} and corollary \ref{cor2} imply $$\frac{M_f(r)}{r^d}\equiv Constant.$$ By the proof of theorem \ref{thm-10} and the definitions of $F$ and $G$, we find that $G(x)-F(x)$ has maximum $0$ on every geodesic ball $\partial B(p, r)$, say at $q(r)$.  By checking the equality of the Hessian comparison (theorem \ref{thm3}), we have that $R_{YJYJYY} = 0$ where $Y =  \nabla r$ at $q(r)$.  From the Taylor expansion of $f$, it is easy to see that the subsequential limit of $Y$ is on the tangent subspace of $p$ spanned by $X$ and $JX$, as $r\to 0^+$. Thus  $R_{XJXJXX} = 0$. This completes the proof of the rigidity.

\section{\bf{Asympototic behavior of holomorphic functions of polynomial growth}}
In this section, we prove that on K\"ahler manifolds with nonnegative bisectional curvature, any holomorphic function with polynomial growth is homogenous at infinity.
The proof involves the convergence of Riemannian manifolds which were established in \cite{[G]}, \cite{[CC1]}\cite{[CC2]}\cite{[CC3]}\cite{[CC4]} and \cite{[Che]}.
We shall review some background here.

Let $(M^n_i, y_i, \rho_i)$ be a sequence of pointed complete Riemannian manifolds, where $y_i\in M^n_i$ and $\rho_i$ is the metric on $M^n_i$. By Gromov's compactness theorem, if $(M^n_i, y_i, \rho_i)$ have a uniform lower bound of the Ricci curvature, then a subsequence converges to some $(M_\infty, y_\infty, \rho_\infty)$ in the Gromov-Hausdorff topology. See \cite{[G]} for the definition and basic properties of Gromov-Hausdorff convergence.
\begin{definition}
Let $K_i\subset M^n_i\to K_\infty\subset M_\infty$ in the Gromov-Hausdorff topology. Assume $\{f_i\}_{i=1}^\infty$ are functions on $M^n_i$, $f_\infty$ is a function on $M_\infty$.  
$\Phi_i$ are $\epsilon_i$-Gromov-Hausdorff approximations, $\lim\limits_{i\to\infty} \epsilon_i = 0$. If $f_i\circ \Phi_i$ converges to $f_\infty$ uniformly, we say $f_i\to f_\infty$ uniformly over $K_i\to K_\infty$.
\end{definition}
 In many applications, $f_i$ are equicontinuous. The Arzela-Ascoli theorem applies to the case when the spaces are different.  When $(M_i^n, y_i, \rho_i)\to (M_\infty, y_\infty, \rho_\infty)$ in the Gromov-Hausdorff topology, any bounded, equicontinuous sequence of functions $f_i$ has a subsequence converging uniformly to some $f_\infty$ on $M_\infty$.

As in section $9$ of \cite{[Che]}, we have the following definition.
\begin{definition}
If $\nu_i, \nu_\infty$ are Borel regular measures on $M_i^n, M_\infty$, we say that $(M^n_i, y_i, \rho_i, \nu_i)$ converges  to $(M_\infty, y_\infty, \rho_\infty, \nu_\infty)$ in the measured Gromov-Hausdorff sense, if $(M^n_i, y_i, \rho_i, \nu_i)\to (M_\infty, y_\infty, \rho_\infty, \nu_\infty)$ in the Gromov-Hausdorff topology and for any $x_i\to x_\infty$ ($x_i\in M_i, x_\infty\in M_\infty$), $r>0$, $\nu_i(B(x_i, r))\to \nu_\infty(B(x_\infty, r))$.
\end{definition}
For any sequence of manifolds with Ricci curvature lower bound, after suitable renormalization of the volume, there is a subsequence converging in the measured Gromov-Hausdorff sense. If the volume is noncollapsed, $\nu_\infty$ is just the $n$-dimensional Hausdorff measure of $M_\infty$. See \cite{[CC2]}.

For a Lipschitz function $f$ on $M_\infty$, define a norm $||f||^2_{1, 2} = ||f||^2_{L^2}+\int_{M_\infty}|Lip f|^2$, where $$Lip(f, x) =\lim\sup\limits_{y\to x}\frac{|f(y)-f(x)|}{d(x, y)}.$$
In \cite{[Che]}, a Sobolev space $H_{1, 2}$ is defined by taking the closure of the norm $||\cdot||_{1, 2}$ for Lipschitz functions.

\emph{Condition (1)}:
$M_\infty$ satisfies the volume doubling property if for any $r>0$, $x\in M_\infty$,  $\nu_\infty(B(x, 2r))\leq 2^n\nu_\infty(B(x, r))$.

\emph{Condition (2)}:
$M_\infty$ satisfies the weak Poincare inequality if $$\int_{B(x, r)}|f - \overline{f}|^2 \leq C(n)r^2\int_{B(x, 2r)}|Lip f|^2$$ for all Lipschitz functions.
Here $\overline f$ is the average of $f$ on $B(x, r)$.

In theorem $6.7$ of \cite{[CC4]}, it was proved that if $M_\infty$ satisfies the $\nu$-rectifiability condition, condition (1) and condition (2), then there is a unique differential $df$ for $f\in H_{1, 2}$.
If $f$ is Lipschitz, $\int |Lipf|^2 = \int |df|^2$. Moreover, the $H_{1, 2}$ norm becomes an inner product. Therefore $H_{1, 2}$ is a Hilbert space. Then there exists a unique self-adjoint operator $\Delta$ on $M_\infty$ such that $$\int_{M_\infty} <df, dg> = \int _{M_\infty}<\Delta f, g>$$ for all Lipschitz functions on $M_\infty$ with compact support (Of course we can extend the functions to Sobolev spaces). See theorem $6.25$ of \cite{[CC4]}.

If $M_i\to M_{\infty}$ in the measured Gromov-Hausdorff sense and that the Ricci curvature is nonnegative for all $M_i$, then the $\nu$-rectifiability of $M_\infty$ was proved in theorem $5.5$ in \cite{[CC4]}.
By the volume comparison, Condition (1) obviously holds for $M_\infty$. Condition (2) also holds. See \cite{[X]} for a proof.

In \cite{[Di1]}\cite{[X]}, the following lemma was proved:
\begin{lemma}\label{lemma-10}
Suppose $M_i$ has nonnegative Ricci curvature and $M_i\to M_\infty$ in the measured Gromov-Hausdorff sense. Let $f_i$ be Lipschitz functions on $B(x_i, 2r)\subset M_i$ satisfying $\Delta f_i = 0$; $|f_i|\leq L, |\nabla f_i|\leq L$ for some constant $L$. Assume $x_i\to x_\infty$, $f_i\to f_\infty$ on $M_\infty$. Then $\Delta f_\infty = 0$ on $B(x_\infty, r)$.
\end{lemma}

\begin{definition}
Let $M$ be a complete K\"ahler manifold and $f\in \mathcal{O}(M)$. The order at infinity is defined by $\overline{\lim\limits_{r\to\infty}}\frac{\log M_f(r)}{\log r}$, where $r$ is the distance to a fixed point $p$ on $M$, $M_f(r)$ is the maximal modulus of $f$ on $B(p, r)$.
\end{definition}
By corollary \ref{cor-7}, if $M$ have nonnegative holomorphic sectional curvature and $f\in \mathcal{O}(M)$ with order $d$ at infinity, then $f\in \mathcal{O}_d(M)$.
Now we can state the main result in this section:
\begin{theorem}\label{thm-99}
Let $M^n$ be a complete noncompact K\"ahler manifold with nonnegative bisectional curvature, $p\in M$, $r(x) =d(x, p)$.  Let $f\in \mathcal{O}_d(M)$ and suppose $d$ is the order at infinity. Then given any $K>1, \epsilon>0$,  there exists $R>0$ such that for any $r>R$, $x\in B(p, Kr)\backslash B(p, r)$, any point $y$ on the minimal geodesic connecting $p$ and $x$, we have $$|f(y)(r(x))^d-f(x)(r(y))^d|\leq \epsilon M_f(r)r^d.$$ In particular, if $|f(x)|\geq \delta M_f(r)$, $$|\frac{f(y)}{f(x)}-(\frac{r(y)}{r(x)})^d|\leq\frac{\epsilon}{\delta}.$$
\end{theorem}
\begin{remark}
This theorem might be related with Yau's celebrated conjecture on the finite generation on holomorphic functions with polynomial growth.
See page $3-4$ for the explanation. \end{remark}
\begin{cor}\label{cor-70}
Let $M^n$ be a complete K\"ahler manifold with nonnegative bisectional curvature. Suppose $M$ is of maximal volume growth and $f\in \mathcal{O}_d(M)$. Here $d>0$ and $d$ is the order at infinity. Consider any tangent cone $C$ at infinity,  say $ds^2_{C} = dr^2+r^2d\Sigma$ for $r\in [0, \infty)$, where $\Sigma$ is a $(2n-1)$-dimensional metric measured space. Then for any positive integer $i$, $\alpha_i = di(2n+di-1)$ is an eigenvalue for the Laplacian operator on $\Sigma$ with respect to the $(2n-1)$-dimensional Hausdorff measure. 
\end{cor}
\begin{proof}
Take $g_i = f^{i}$ for $i\in\mathbb{N}$. The proof follows from (\ref{eq-70}) and Claim \ref{cl-5} below.
\end{proof}
\begin{remark}
There is a conjecture of Ni \cite{[N1]} stating that if $M$ has nonnegative bisectional curvature and maximal volume growth, then there exists a nontrivial holomorphic function with polynomial growth on $M$. In view of the corollary, it is possible that the tangent cone of a K\"ahler manifold $M$ is unique, if $M$ has nonnegative bisectional curvature and maximal volume growth. In the Riemannian case, when the Ricci curvature is nonnegative and the manifold has maximal volume growth, the tangent cone is not unique in general. Perelman constructed the first example. Moreover, different tangent cones may not be isospectral, see example 11.4 in \cite{[CM1]} .
\end{remark}

\begin{remark}
There are ``many" points $q$ in $B(p, Kr)\backslash B(p, r)$ satisfying $|f(q)|> \delta M_f(r)$. In fact, any geodesic ball with radius $\delta^\frac{1}{d}Kr$ contains such points. Here we are using corollary \ref{cor1}.
\end{remark}
\begin{remark}
We can weaken the condition in theorem \ref{thm-99} by assuming $M$ has nonnegative Ricci curvature and nonnegative holomorphic sectional curvature, if $M$ has maximal volume growth.
\end{remark}
\begin{remark}
In theorem \ref{thm-99}, we cannot make $R$ uniform for all holomorphic functions with polynomial growth. To see this, just consider $f_t(z) = z^n+ tz^{n-1}$ on $\mathbb{C}$.
\end{remark}
\begin{remark}
In remark \ref{rm1}, we compared the volume comparison theorem with the three circle theorem. 
When a manifold has nonnegative Ricci curvature and maximal volume growth, near infinity, the sharp relative volume comparison almost holds. This explains in some sense why the almost-warped product theorem of Cheeger and Colding \cite{[CC1]} should hold. Now assume $M$ has nonnegative holomorphic sectional curvature and $f\in \mathcal{O}_d(M)$ (say $d>0$ is the order at infinity). Then for any $k>1$, by the three circle theorem, $\frac{M_f(kr)}{M_f(r)}$ is monotonic increasing and $\lim\limits_{r\to\infty}\frac{M_f(kr)}{M_f(r)}= k^d$. Therefore the equality for the three circle theorem almost holds. We should a have result stating the behavior of $f$ at infinity, just like the almost-warped product theorem of Cheeger and Colding. In theorem \ref{thm-99}, we used the tangent cone at infinity, which requires the nonnegativity of the Ricci curvature.
\end{remark}

\begin{proof}
First we consider the case when the universal cover $\tilde{M}$ does not split as a product of two K\"ahelr manifolds. By the main theorem in \cite{[L]}, $M$ is of maximal volume growth.
Assume there exist $K>1, \epsilon_0>0$; a sequence $r_i\to \infty$; $x_i\in B(p, Kr_i)\backslash B(p, r_i)$; $y_i$ on a minimal geodesic connecting $p$ and $x_i$ so that 
\begin{equation}\label{eq-75}
|f(y_i)(r(x_i))^d-f(x_i)(r(y_i))^d|\geq \epsilon_0 M_f(r_i)r_i^d.
\end{equation}

By the almost warped-product theorem of Cheeger and Colding \cite{[CC1]}, there exists a subsequence of $r_i$ which we still call $r_i$, such that the rescaled metrics $(M, r_i^{-2}g, p)$ pointed converge to a metric cone $C(\Sigma) = (\Sigma\times_{r^2} \mathbb{R}^+, g_\infty, p_\infty)$ in the Gromov-Hausdorff topology. Here $\Sigma$ is a compact metric measured space with Hausdorff dimension $n-1$. Define $g_i = r_i^{-2}g$. By the discussion in the beginning of this section, there is a Laplacian operator on $C(\Sigma)$.  Define a rescaled function 
\begin{equation}\label{eq-76}
f_i(x) = \frac{f(x)}{M_f(r_i)}.
\end{equation} Then $\sup |f_i(x)| = 1$ for $x\in B(p, r_i)$. According to corollary \ref{cor1}, $$|f_i(x)|\leq (2k)^d$$ on $B(p, 2kr_i)$. Here $k>1$. By Cheng-Yau's gradient estimate \cite{[CY]},  $$|\nabla_{g_i}f_i(x)|\leq a(n)k^{d-1}$$ on $B(p, kr_i)$. Here $a(n)$ is a constant depending only on $n$.  Thus there exists a subsequence of $f_i$ converging to $f_\infty$ uniformly on each compact set of the metric cone $C(\Sigma)$.

By theorem \ref{thm-10}, for any $k>1$, $\frac{M_f(kr)}{M_f(r)}$ is monotonic increasing. Since $d = \overline{\lim\limits_{r\to\infty}}\frac{\log M_f(r)}{\log r}$,  $$\lim\limits_{r\to\infty}\frac{M_f(kr)}{M_f(r)} =k^d.$$ Let $M_{f_\infty}(r)$ be the maximal modulus of $f_\infty$ on $B(p_\infty, r)$. Then for any $r_1>r_2>0$,
$$\frac{M_{f_\infty}(r_1)}{M_{f_\infty}(r_2)} = \lim\limits_{i\to\infty}\frac{M_f(r_1r_i)}{M_f(r_2r_i)} = \frac{r_1^d}{r_2^d}.$$ Since $M_{f_\infty}(1) = 1$,
\begin{equation}\label{eq-106}
M_{f_\infty}(r) = r^d.
\end{equation}

$f_i$ are harmonic functions. By Lemma \ref{lemma-10}, $f_\infty$ is harmonic on $C(\Sigma)$. Also it is of polynomial growth. 

It is easy to see that the $(n-1)$ dimensional Hausdorff measure on $\Sigma$ satisfies the volume doubling property and the weak Poincare inequality. See lemma $4.3$ in \cite{[Di1]} for a proof. Also, one can directly check that $\Sigma$ is $\nu$-rectifiable. Therefore, we have a Laplacian operator on $\Sigma$.

On the metric cone $C(\Sigma)$, there is a decomposition formula (see \cite{[Di2]}\cite{[Di1]}).
\begin{equation}\label{eq-99}
\Delta u = -\frac{\partial^2 u}{\partial r^2} -\frac{n-1}{r}\frac{\partial u}{\partial r}+\frac{1}{r^2}\Delta_\Sigma u.
\end{equation}
Therefore, if $\phi_i$ is the $i$-th eigenfunction of $\Delta_\Sigma$ with eigenvalue $\lambda_i$, then $r^{\alpha_i}\phi_i(x)$ is harmonic. Here 
\begin{equation}\label{eq-70}
\lambda_i=\alpha_i(n+\alpha_i-2).
\end{equation} We normalize so that $||\phi_i||_{L^2(\Sigma)} = 1$. For any harmonic function (complex function) $u$ on $X$,  we can write (see \cite{[Che0]}\cite{[Di2]})
\begin{equation}\label{eq-105}
u = \sum\limits_{i=0}^\infty c_ir^{\alpha_i}\phi_i.
\end{equation}
Here $c_i$ are complex constants.
Define $I(r) = \frac{1}{Vol(\partial B(p_\infty, r))}\int_{\partial B(p_\infty, r)}|u|^2$. Then $$I(r) = \sum\limits_{i=0}^\infty |c_i|^2r^{2\alpha_i}.$$
This implies that if $u$ is of polynomial growth on $C(\Sigma)$, there are only finitely many terms in (\ref{eq-105}).
\begin{claim}\label{cl-5}
$f_\infty = r^d\phi$ for some $\phi$ on $\Sigma$ with $\Delta_\Sigma\phi = d(n+d-2)\phi$.
\end{claim}
\begin{proof}
Since $f_\infty$ is harmonic, by (\ref{eq-105}) and (\ref{eq-106}), if $r\to 0$, we find $\alpha_i\geq d$ for all $i$. If $r\to \infty$, we find $\alpha_i\leq d$ for all $i$.
\end{proof}

According to our assumption of $x_i$ and $y_i$, there exists a subsequence such that $x_i \to x_\infty, y_i\to y_\infty$ with $x_\infty, y_\infty \in B(p_\infty, K)$; $cx_\infty = y_\infty$ for some constant $c$ between $0$ and $1$. The multiplication is for the factor $R^+$ on the metric cone $C(\Sigma)$.  By Claim \ref{cl-5}, $$f_\infty(x_\infty)(r(y_\infty))^d =f_\infty(y_\infty)(r(x_\infty))^d.$$
By (\ref{eq-76}),  for large $i$, $$|f(x_i)(r(y_i))^d - f(y_i)(r(x_i))^d|<\frac{\epsilon_0}{2}r_i^dM_f(r_i).$$
This contradicts (\ref{eq-75}).

\bigskip

Next we consider the case when the universal cover of $M$ has splitting factors. Assume $\tilde{M} = \Pi_{i}M_i\times\Pi_jN_j$ where $M_i, N_j$ do not have splitting factors; $M_j$ has nonconstant holomorphic functions with polynomial growth; $N_j$ does not admit nonconstant holomorphic functions with polynomial growth.  Let $M' = \Pi_iM_i, N' =\Pi_jN_j$. Since each factor of $M'$ has maximal volume growth, $M'$ has maximal volume growth condition. Therefore $M'$ satisfies theorem \ref{thm-99}.

Suppose theorem \ref{thm-99} does not hold on $\tilde{M}$ for some sequence $r_i\to\infty$. Then there exist $\epsilon_0>0$; $K>1$; $f\in\mathcal{O}_d(\tilde{M})$ where $d$ is the order at infinity; $x_i\in B(p, Kr_i)\backslash B(p, r_i)$ of $\tilde{M}$; $y_i$ lies on the minimal geodesic connecting $p$ and $x_i$ so that
\begin{equation}\label{eq-108}
|f(y_i)(r(x_i))^d-f(x_i)(r(y_i))^d|\geq \epsilon_0 M_f(r_i)r_i^d.\end{equation}
Note that $f$ is constant on $N'$.
We can project $p, x_i, y_i$ to $M'$, say the images are $p', x_i', y_i'$. Then $y_i'$ lies on the minimal geodesic connecting $p'$ and $x_i'$.   For $x'\in M'$, define $r(x') = d(p', x')$. Let $f'$ be the restriction of $f$ on $M'$, $M_{f'}(r)$ be the maximal modulus of $f'$ on $B(p', r)$ in $M'$.
\begin{claim}\label{cl21}
For sufficiently large $i$,  $r(x_i')\geq (\epsilon_0)^{\frac{1}{d}}\frac{r_i}{4K}$. 
\end{claim}
\begin{proof}

Observe $f(x_i) = f(x_i')$; $f(y_i)=f(y_i')$; $M_f(r) = M_{f'}(r)$ for all $r>0$. Let $\delta = (\epsilon_0)^{\frac{1}{d}}\frac{1}{4K}$.  By the three circle theorem, $\frac{M_{f'}(r)}{M_{f'}(\delta r)}$ is monotonic increasing, and $$\lim\limits_{r\to\infty}\frac{M_{f'}(r)}{M_{f'}(\delta r)}= \frac{1}{\delta^d}.$$ Then there exists $R$ such that when $r>R$, $\frac{M_{f'}(r)}{M_{f'}(\delta r)}>\frac{1}{2\delta^d}.$ For sufficiently large $i$, $r_i>R$. Therefore $$M_{f'}(\delta r_i)<2\delta^d M_{f'}(r_i).$$ Suppose $r(x_i')<\delta r_i$ for some large $i$. Therefore $x_i', y_i' \in B(p', \delta r_i)$ and $$|f(x_i')|, |f(y_i')|\leq M_{f'}(\delta r_i)< 2\delta^dM_{f'}(r_i) = 2\delta^d M_f(r_i).$$  Plugging this in (\ref{eq-108}), we find a contradiction.
\end{proof}

A subsequence of $x_i', y_i'$ in  $(M', r_i^{-2}g, \frac{\nu}{Vol(B(p', r_i))})$ converge to some $x_\infty', y_\infty'$. By Claim \ref{cl21}, $x_\infty'$ is away from the pole. Also $\frac{r(x_i)}{r(x_i')} = \frac{r(y_i)}{r(y_i')}\leq \frac{K}{(\epsilon_0)^{\frac{1}{d}}\frac{1}{4K}}$. By (\ref{eq-108}), 
\begin{equation}\label{eq-109}
|f(y_i')(r(x_i'))^d-f(x_i')(r(y_i'))^d|\geq \epsilon_0 M(r_i)r_i^d\cdot((\epsilon_0)^{\frac{1}{d}}\frac{1}{4K^2})^d = \frac{\epsilon_0^2}{4^dK^{2d}}M_{f'}(r_i)r_i^d.\end{equation}

 By a limit argument as before, we find (\ref{eq-109}) contradicts that $M'$ satisfies theorem \ref{thm-99}. 
Finally, to show theorem \ref{thm-99} holds for $M$, it suffices to lift the minimal geodesic connecting $p$ and $x$ to $\tilde{M}$.

\end{proof}
\section{\bf{A Liouville type theorem for plurisubharmonic functions}}
The classical Liouville theorem states that any bounded (or even
just positive) harmonic function is constant on Euclidean space. In \cite{[Y]}, Yau extended the
classical Liouville theorem to complete Riemannian manifolds with nonnegative
Ricci curvature. It was further shown by Cheng and Yau \cite{[CY]} that any harmonic
function with sublinear growth on a complete Riemannian manifold with
nonnegative Ricci curvature must be constant.

In the K\"ahler category, it is easy to see that any non-constant plurisubharmonic function in $\mathbb{C}^n$ grows faster than logarithmic function. In \cite{[N2]}, Ni showed that on a complete K\"ahler manifold $M^n$ with nonnegative Ricci curvature, any plurisubharmonic function $f$ with sublogarithmic growth has $(\partial\overline\partial f)^n = 0$.
 In \cite{[NT]}, Ni and Tam proved that any plurisubharmonic function is constant, provided it is of sublogarithmic growth and the manifold has nonnegative bisectional curvature.  Their method combines the heat flow and a result in \cite{[N2]}. We show this Liouville type theorem still holds when the holomorphic sectional curvature is nonnegative.
\begin{theorem}\label{thm4}
Let $M^n$ be a complete noncompact K\"ahler manifold with nonnegative holomorphic sectional curvature. Let $p\in M$ and $r(x) = dist(x, p)$. Assume $u$ is a function on $M$ satisfying $\overline{\lim\limits_{r\to \infty}}\frac{u^+(x)}{\log r} = 0$. Here $u^+(x) = max(u(x), 0)$. 
\begin{itemize}
\item
 If $u$ be a plurisubharmonic function(not necessarily continuous) on $M$, then $u$ is a constant.  
 \item
 If $u$ is smooth satisfying $u_iu_{\overline{j}}u_{k\overline{l}}g^{i\overline{l}}g^{k\overline{j}}\geq 0$, then $u$ is a constant. This condition means the complex Hessian of $u$ is nonnegative along the gradient of $u$.
 \end{itemize}
\end{theorem}
\begin{remark}
If we define $M(r) = \max\limits_{r(x)\leq r}u^+(x)$ and assume $\lim \limits_{\overline{r\to \infty}}\frac{M(r)}{\log r} = 0$, then theorem \ref{thm4} still holds. In fact, theorem \ref{thm4} holds even we assume that the holomorphic sectional curvature is asymptotically nonnegative. See corollary \ref{cor5}.\end{remark}
\begin{proof}
Let $A = \sup\limits_{B_p(1)} u(x)$. Given any $\epsilon>0, 1>\epsilon_1>0$, define $$f_{\epsilon_1}(x)=A+\epsilon\log\frac{r-\epsilon_1}{1-\epsilon_1}; h(x)=u(x)-f_{\epsilon_1}(x).$$  Assume $h(x)$ achieves the maximum inside the annulus $B(p, R)-B(p, 1)$, say at $q$. Consider the small holomorphic disk $D$ centered at $q$ such that $T_qD$ is spanned by  $\frac{\partial}{\partial r}$ and $J\frac{\partial}{\partial r}$. By theorem \ref{thm3}, $$h_{1\overline1} >0$$ where $e_1 = \frac{1}{\sqrt{2}}(\nabla r - J\nabla r)$. 
Then for small $a>0$,
$$h(q) < \frac{1}{\pi a^2}\int\limits_{D(q, a)}h(x)dx\wedge dy.$$ Here $D(q, a)$ is the holomorphic disk in $D$ with radius $a$; $x+iy$ is the complex coordinate on $D$. This contradicts the assumption that $h$ has the maximum at $q$. If $q$ is on the cut locus of $p$, we can apply Calabi's trick as before. The details are omitted. Note that $h(x) \leq 0$ for $x\in \partial B(p, 1)\bigcup \partial B(p, R)$ for large $R$.  Therefore, on $B(p, R)-B(p, 1)$, $h\leq 0.$ Letting $R\to\infty$, $\epsilon\to 0$ and $\epsilon_1\to 0$, we obtain $$u(x)\leq A.$$ Since $A$ is the maximum of $u(x)$ in $B(p, 1)$, maximum principle says $u(x) \equiv A$.

\bigskip

Now consider the case when $u_iu_{\overline{j}}u_{k\overline{l}}g^{i\overline{l}}g^{k\overline{j}}\geq 0$.  Let $A = \inf u$ on $M$ (here $A$ could be $-\infty$). Consider a decreasing sequence $a_i\to A$. We pick $x_i\in M$ such that  $u(x_i)\leq a_i$. Let $b_i = a_i+\frac{1}{i}$, then there exists $\delta_i>0$ such that $u(x)\leq b_i$ for $x\in B(x_i, \delta_i)$. Define $$f_i(x) = b_i+\epsilon\log\frac{dist(x_i, x)}{\delta_i}$$(Here $\epsilon$ is a small positive constant). Then $u(x)\leq f_i(x)$ for $x\in \partial B(x_i, \delta_i)$ and $x\in \partial B(x_i, R)$ for $R$ sufficiently large. If $u-f_i(x)$ achieves the maximum at $q$ in the interior part of $B(x_i, R)-B(x_i, \delta)$, then $$\nabla u(q) = \nabla f_i(q) = C\nabla dist(x_i, x)|_{x=q}.$$ If $q$ is on the cut locus of $x_i$, we apply Calabi's trick.  By maximum principle as before, $$f_i(x)\geq u(x)$$ for $x\in M-B(x_i, \delta_i)$.  Let $\epsilon\to 0$ and then $i\to\infty$, we find that $$u(x)\leq A.$$ Since $A = \inf u$, $u\equiv A$.

\end{proof}
\section{\bf{Sharp dimension estimate on bundles}}

\begin{definition}
Let $E$ be a Hermitian holomorphic vector bundle on a K\"ahler manifold $M$. We say $E$ is nonnegative (nonpositive) in the sense of Griffith if $\Theta(F) ( \xi\otimes v)\geq 0$  ($\leq 0$) for all nonzero indecomposable tensor $\xi\otimes v \in TM\otimes F$. Here $\Theta$ is the curvature of the Chern connection of $E$, $F$ is the fibre of $E$.
\end{definition}
\begin{definition}
Let $E$ be a Hermitian holomorphic vector bundle over a K\"ahler manifold $M$.  Let $\mathcal{O}_M(E)$ be holomorphic sections on $E$. For any $d\geq 0$, define $\mathcal{O}_d(M, E) = \{f\in \mathcal{O}_M(E)|\overline{\lim\limits_{r\to\infty}}\frac{|f(x)|}{r^d}<\infty\}$. Here $r$ is the distance from a fixed point on $M$. If $f\in \mathcal{O}_d(M, E)$, we say $f$ is of polynomial growth with order $d$.
\end{definition} 

\begin{theorem}\label{thm5}
Let $M^n$ be a complete noncompact K\"ahler manifold with nonnegative holomorphic sectional curvature. If $E^m$ is a Hermitian holomorphic vector bundle of rank $m$ over $M$ such that $E$ is nonpositive in the sense of Griffith, then for any $d > 0$, 
\begin{equation}\label{eq-28}
dim(\mathcal{O}_d(M, E)) \leq dim(\mathcal{O}_d(\mathbb{C}^n, E'))
\end{equation} where $E'$ is the trivial flat holomorphic vector bundle of rank $m$ over $\mathbb{C}^n$. If the equality holds for some integer $d\geq 1$, $(M, E)$ is biholomorphic and isometric to $(\mathbb{C}^n, E')$.
\end{theorem} 
\begin{remark}
This type of theorem was proved by Ni \cite{[N1]} when $E$ is a nonpositive line bundle and $M$ has nonnegative bisectional curvature.
\end{remark}
\begin{proof}
For any holomorphic section $f\in \mathcal{O}_M(E)$,  the Poincare-Lelong equation says $$\frac{\sqrt{-1}}{2\pi}\partial\overline\partial\log |f|^2 =-\Theta(L_f).$$ Here $L_f$ is the line bundle induced by the section $f$ at the points where $f\neq 0$, $\Theta(L_f)$ is the curvature form of $L_f$. It is well known that the curvature of a Hermitian subbundle is no greater than the curvature of the ambient bundle \cite{[GH]}. Therefore, $$\frac{\sqrt{-1}}{2\pi}\partial\overline\partial\log |f|^2 \geq 0$$ when $f\neq 0$. In particular, $$\Delta(|f|^2)\geq 0,$$ thus $|f|^2$ cannot assume the maximum in the interior part of a domain unless it is a constant. 

By similar arguments in the proof of theorem \ref{thm-10},  we obtain the three circle theorem for $|f|$. Then corollary \ref{cor1} and corollary \ref{cor2} hold. This proves the sharp dimension estimate.
\bigskip

Now we turn to the rigidity.   Fix a point $p\in M$. Let $e_1, e_2,...e_m$ be a local holomorphic basis for $E$ with $\nabla e_i = 0$ at $p$ for $i=1, 2,...m$. Consider a holomorphic chart  $(U, z_1, z_2,...z_n)$ containing $p$.  To show $M$ is flat and $E$ is a flat bundle over $M$, it suffices to show $R_{1\overline{1}1\overline{1}} = 0$ and $\Theta_{1\overline{1}}(e_1, \overline{e_1}) = 0$ at $p$ (since the frame is arbitrary). Here $\Theta$ is the curvature of $E$.  Assume the equality holds in (\ref{eq-28}) for some positive integer $d$.  There exists $f\in \mathcal{O}_d(M, E)$ such that $$f(z_1, z_2, \cdot\cdot\cdot, z_n) = z_1^de_1+\sum\limits_{i=1}^mO(r^{d+1})e_i.$$ Similar to the proof of rigidity of theorem \ref{thm2}, we get that $$\frac{M_f(r)}{r^d}\equiv Constant.$$ Here $M_f(r) = max |f(x)|$, $x\in B(p, r)$.  Then by similar arguments in the proof of theorem \ref{thm-10} (note that the Poincare-Lelong equation gives the curvature), we find that $$R_{YJYJYY} = 0; \Theta_{\alpha\overline{\alpha}}(L_f)(f(x), \overline{f(x)}) = 0$$ at the points on $\partial B(p, r)$ where $|f(x)|$ takes the maximum.   Here $Y = \nabla r$,  $\alpha= \frac{1}{2}(Y-\sqrt{-1}JY)$. Since
$$0\geq\Theta_{\alpha\overline{\alpha}}(f(x), \overline{f(x)})\geq\Theta_{\alpha\overline{\alpha}}(L_f)(f(x), \overline{f(x)}) = 0,$$  $$\Theta_{\alpha\overline{\alpha}}(f(x), \overline{f(x)})=0.$$ Letting $r\to 0$, we find $$R_{1\overline{1}1\overline{1}} = 0; \Theta_{1\overline{1}}(e_1, \overline{e_1}) = 0$$ at $p$. Therefore both $M$ and $E$ are flat. By looking at the pull back bundle of the covering map $\mathbb{C}^n\to M$ and counting dimensions, we find that $M$ is isometric and biholomorphic to $\mathbb{C}^n$.

\end{proof}

\section{\bf{Holomorphic maps between certain noncompact K\"ahler manifolds}}

Given a complex analytic space $X$, the Douady space \cite{[Do]} parametrizes all pure dimensional compact analytic subspace of $X$. When $X$ is projective, the Douady space is exactly the Hilbert Scheme of $X$, as constructed by A. Grothendieck \cite{[Gr]}. The Douady space has a natural complex structure. 

Below we focus on the set $H$ of holomorphic maps from  a complex manifold $M$ to a complex manifold $N$. If $M$ is compact, $H$ is an open set of the Douady space of $M\times N$ \cite{[Do]}.  Thus $H$ a complex analytic space with locally finite dimension.

When $M$ is noncompact, the situation is more complicated. $H$ might have infinite dimension even locally. For example, the space of holomorphic maps from $\mathbb{C}^m$ to $\mathbb{C}^n$ has infinite dimension. However, if we impose a polynomial growth condition, then the dimension is finite. Even with the growth condition, we have to impose some geometric conditions on $M$ and $N$ to ensure that finiteness of the dimension. 
\begin{definition}\label{def0}
Let $N^n$ be complex manifolds. Let $o$ be a point in $\mathbb{C}^m$ and $k$ be a positive integer. For any local holomorphic map $f$ from $U\ni o$  to $N$, define the $k$-jet of the map $f$ as $(f(o), df(o), d^2f(o), \cdot\cdot\cdot, d^kf(o))$. Here $d^kf(o)$ are the partial derivatives with order $k$ in some holomorphic chart containing $f(o)$. Define $J^k_mN =\{(f(o), df(o), d^2f(o), \cdot\cdot\cdot, d^kf(o))\}$ for all holomorphic maps $f$ from a neighorbood of $o$ to $N$. Then it is easy to check that $J^k_mN$ is a complex manifold.
\end{definition}

\begin{definition}
Let $M$ and $N$ be complete complex manifolds. Let $o$ be a point on $M$ and define $r(x) = dist_M(o, x)$.
We say a holomorphic map $f$ from $M$ to $N$ is of polynomial growth with order $k$, denoted by $f\in \mathcal{O}_k(M, N)$, if 
\begin{equation}
\overline{\lim\limits_{r\to \infty}}\frac{dist_N(f(o), f(x))}{r^k}  < \infty.
\end{equation}
We put the compact open topology on $\mathcal{O}_k(M, N)$.\end{definition}

The main result in this section is the following:
\begin{theorem}\label{thm6}
Let $M^m$ be a complete K\"ahler manifold with nonnegative holomorphic sectional curvature and $N^n$ be a simply connected K\"ahler manifold with nonpositive sectional curvature. Let $o$ be a fixed point on $M$ and define $r(x) = dist_M(o, x)$.
For any $k\geq 1$, let $H_k = \mathcal{O}_k(M, N)$. 
 Then the jet map $\iota$: $H_k\to J^{[k]}_mN$ defined by $\iota(f) = (f(o), df(o), d^2f(o), \cdot\cdot\cdot, d^{[k]}f(o))$ is an embedding. Now assume $k$ is a positive integer. Then $\iota$ is a homeomorphism if and only if $M^m, N^n$ are isometric and biholomorphic to complex Euclidean spaces.
\end{theorem}
\begin{remark}
It is highly possible that the image of $H_k$ in $J^{[k]}_mN$ is a complex analytic subspace.
\end{remark}
\begin{proof} We first introduce the following proposition:
\begin{prop}\label{prop-9}
Let $K$ be a compact set in $J^{[k]}_mN$. Define $H_{k, K} = \iota^{-1}(K)$. Then $H_{k, K}$ is compact.
\end{prop}
\begin{proof}
 For notational convenience, let $d(f(x), f(o)) = dist_N(f(x), f(o))$. 
\begin{lemma}\label{lm-7}
 $\log d(f(o), f(x))$ is plurisubharmonic.  
 \end{lemma} 
 \begin{proof}
 Consider a unitary frame $e_i$ at $x$ and a holomorphic chart $(V, w_1, \cdot\cdot\cdot, w_n)$ containing $f(x)$ such that $\frac{\partial}{\partial w_\alpha} =  e_\alpha$ has unit length at $f(x)$.
 \begin{equation}
\begin{aligned}
 (\log d(f(o), f(x)))_{i\overline{i}} &= (\frac{d_\alpha\alpha_i}{d})_{\overline{i}} \\&= \frac{d_{\alpha\overline{\beta}}\overline\beta_{\overline{i}}\alpha_i}{d}+\frac{d_\alpha\alpha_{i\overline{i}}}{d}-|\frac{d_\alpha\alpha_i}{d}|^2\\&\geq 0 .
 \end{aligned}
 \end{equation} 
 Here we have used the Hessian comparison theorem $d_{\alpha\overline{\beta}}\geq \frac{1}{2d}\delta_{\alpha\beta}$; $\alpha_{i\overline{i}} = 0$ ; $|d_\alpha|\leq \frac{1}{\sqrt{2}}$. 
 \end{proof}
 
 Let $M(f, r)$ be the maximum of $d(f(o), f(x))$ for $x\in B(o, r)$.  By lemma \ref{lm-7} and similar arguments as before, we have
 \begin{claim}\label{cl-3}
$\frac{M(f, r)}{r^k}$ is nonincreasing. 
\end{claim}
Consider a sequence $f_j\in H_{k, K}$. We may assume $f_j(o)\to p\in N$. Consider a small holomorphic chart $U(o, z_1, z_2,...z_m)$ in $M$ such that for $x = (z_1,...,z_m) \in U$,
 \begin{equation}\label{1}
 \frac{1}{2}\sqrt{\sum\limits_{i=1}^m|z_i|^2}\leq |x|=dist(o, x)\leq 2\sqrt{\sum\limits_{i=1}^m|z_i|^2}.
 \end{equation} 
 
By scaling the metric on $M$, we may assume $U$ contains the points $z= (z_1, z_2,...., z_m)$ where $|z| = \sqrt{\sum\limits_{i=1}^m|z_i|^2}\leq 1$.

Case 1:
There exist $\delta > 0$, $M > 0$ and a subsequence $f_{j_s}$ so that $dist(f_{j_s}(o), f_{j_s}(x))\leq M$ for all $x\in B(o, \delta)$. By claim \ref{cl-3}, for any $r\geq \delta$, $$M(f_{j_s}, r)\leq r^k\frac{M}{\delta^k}.$$
Since $N$ is simply connected with nonpositive sectional curvature, $N$ is a Stein manifold. Therefore $N$ could be properly holomorphically embedded in $\mathbb{C}^{2n+1}$. Then there exist uniform bounds of the derivatives of $f_{j_s}$ in each compact set of $M$.
Thus there exists a subsequence of $f_{j_s}$ converging to some $f\in H_{k, K}$.

\bigskip

Case 2: Case 1 does not hold. 
Then for large $j$, there exists $\delta_j \to 0$ satisfying $$M(f_j(\delta_jz), |z|\leq \frac{1}{4}) = 1.$$ Here $z$ is in $U$, $\delta z = \delta\cdot(z_1, \cdot\cdot\cdot, z_m) = (\delta z_1, \cdot\cdot\cdot, \delta z_m)$. 
Define $g_j (z)= f_j(\delta_jz)$. Note $g_j$ is only defined on $U$.
\begin{lemma}\label{lem-3}
For $\frac{1}{3}\geq r_1 \geq r_2>0$, $\frac{M(g_j, |z|\leq r_1)}{4^kr_1^k}\leq \frac{M(g_j, |z|\leq r_2)}{r_2^k}$.
\end{lemma}
\begin{proof}
We have \begin{equation}
\begin{aligned}\frac{M(g_j, |z|\leq r_1)}{2^k\delta_j^kr_1^k}&= \frac{M(f_j(\delta_jz), |z|\leq r_1)}{2^k\delta_j^kr_1^k}\\&\leq \frac{M(f_j, 2\delta_jr_1)}{(2\delta_jr_1)^k}\\&\leq \frac{M(f_j, \frac{1}{2}\delta_jr_2)}{(\frac{1}{2}\delta_jr_2)^k}\\&\leq \frac{M(f_j(\delta_jz), |z|\leq r_2)}{(\frac{1}{2}\delta_jr_2)^k}\\&=\frac{M(g_j, |z|\leq r_2)}{(\frac{1}{2})^k\delta_j^kr_2^k}.\end{aligned}\end{equation} In the middle, we have used claim \ref{cl-3} and (\ref{1}).
\end{proof}
Now $g_j(0) \to p$ and $M(g_j, |z|\leq \frac{1}{4}) = 1$. By lemma \ref{lem-3}, $M(g_j, |z|\leq\frac{1}{3})$ is uniformly bounded. Then we can find a subsequence such that $g_j\to g$ uniformly for $|z|\leq \frac{1}{4}$. Thus $g$ is not a constant map.
Uniform convergence and lemma \ref{lem-3} imply the following:
\begin{lemma}\label{lem-4}
For $\frac{1}{4}\geq r_1 \geq r_2>0$, $\frac{M(g, |z|\leq r_1)}{4^kr_1^k}\leq \frac{M(g, |z|\leq r_2)}{r_2^k}$. 
\end{lemma}
Since $K$ is compact in $J^{[k]}_mN$, the partial derivatives of $f_j$ at $o$ are uniformly bounded up to order $[k]$. As $\delta_j\to 0$, the partial derivatives of $g_j$ at $o$ converge to $0$ up to order $[k]$. Therefore, the partial derivatives of $g$ at $o$ vanish up to order $[k]$. Combining this with lemma \ref{lem-4}, we find $g$ to be a constant map (let $r_2\to 0$ in lemma \ref{lem-4}). This is a contradiction.
Thus Case 2 cannot happen! This completes the proof of proposition \ref{prop-9}.
\end{proof}
\begin{lemma}\label{lem1}
Let $M^m$ be a complete K\"ahler manifold and $N^n$ be a simply connected K\"ahler manifold with nonpositive sectional curvature. Let $f, g$ be holomorphic maps from $M$ to $N$. Then $\log d(f(x), g(x))$ is a plurisubharmonic function on $M$.
\end{lemma}
\begin{proof}
We may assume that at a point $x$, $f(x)\neq g(x)$.
 Let $\gamma$ be the geodesic connecting $f(x)$ and $g(x)$ and $e_\alpha$ $(\alpha = 1, 2,..., 2n)$ be a parallel orthonormal frame on $\gamma$. Let $e_1$ be tangential to the $\gamma$. We also assume $Je_{2s-1} = e_{2s}$ for $s=1, 2,..., n$. Define $h_s = \frac{1}{\sqrt{2}}(e_{2s-1}-\sqrt{-1}e_{2s})$. Near $x\in M$, consider a normal coordinate $(U, z_1,..., z_m)$. 
 \begin{equation}\label{0}
 \begin{aligned}
 \frac{\partial^2 \log d(f(x), g(x))}{\partial z_j\partial{\overline{z_j}}} &= (\frac{(d(f(x), g(x)))_j}{d(f(x), g(x))})_{\overline{j}}\\& = \frac{(d(f(x), g(x))_{j\overline{j}}}{d(f(x), g(x))}-|
\frac{(d(f(x), g(x)))_j}{d(f(x), g(x))}|^2.
\end{aligned}
\end{equation}

We need the second variation of arc length. For completeness, we include the calculations.
Let $F(\lambda, t): [0, 1]\times (-\epsilon, \epsilon)\to N$ be a smooth map. Set $X = F_*{\frac{\partial}{\partial\lambda}}$. We assume that for fixed $t$, $F(\lambda, t)$ is a geodesic in $N$ and $\nabla_XX = 0$.
For each $t$, let $L(t)$ be the arc length $\int\limits_{0}^{1}\sqrt{\langle X, X\rangle}d\lambda$.  To simplify the notation, we write $t = \frac{\partial}{\partial t}$.
First variation at $t=0$:
\begin{equation}\label{2}
\begin{aligned}
\frac{dL(t)}{dt}&=\int\limits_{0}^{1} \frac{\langle\nabla_tX, X\rangle}{\sqrt{\langle X, X\rangle}}d\lambda\\&=\int\limits_{0}^{1} \frac{\langle\nabla_Xt, X\rangle}{\sqrt{\langle X, X\rangle}}d\lambda\\&=\frac{\langle t, X\rangle}{L(0)}|_{\lambda=0}^{\lambda=1}.
\end{aligned}\end{equation}
Second variation at $t=0$:
\begin{equation}\label{3}
\begin{aligned}\frac{d^2L(t)}{dt^2}&=\frac{d}{dt}(\int\limits_{0}^{1} \frac{\langle\nabla_Xt, X\rangle}{\sqrt{\langle X, X\rangle}}d\lambda)\\&=\int\limits_{0}^{1} \frac{\langle\nabla_t\nabla_Xt, X\rangle+|\nabla_Xt|^2}{\sqrt{\langle X, X\rangle}}-\frac{|\langle\nabla_Xt, X\rangle|^2}{{\langle X, X\rangle}^{\frac{3}{2}}}d\lambda\\& =\int\limits_{0}^{1} \frac{R_{tXtX}+\langle\nabla_X\nabla_tt, X\rangle+|\nabla^\perp_Xt|^2}{\sqrt{\langle X, X\rangle}}d\lambda\\&\geq \frac{\int\limits_{0}^{1}|\nabla^\perp_Xt|^2d\lambda+\langle\nabla_tt, X\rangle|_{\lambda=0}^{\lambda=1}}{L(0)}. \end{aligned}\end{equation} Here $\perp$ is the projection orthogonal to $X$.

\bigskip

Let $\gamma_1, \gamma_2$ be two normal geodesics on $M$ starting from $x$, with initial tangent vectors $Re\frac{\partial}{\partial z_i}$ and $Im\frac{\partial}{\partial z_i}$. Let $F_1(\lambda, t), F_2(\lambda, t)$ be two maps satisfying the properties of $F(\lambda, t)$ above. The boundary conditions are $F_j(0, t) = f(\gamma_j(t)), F_j(1, t) = g(\gamma_j(t)), j = 1, 2$. Let $L_j$ be the length function for $F_j$.
\begin{lemma}\label{lm-0}
$\nabla_{t_1}t_1+\nabla_{t_2}t_2 = 0$ for $\lambda = 0$ or $\lambda = 1$. Here $t_j= \frac{\partial}{\partial t_j}, j= 1, 2$.
\end{lemma}
\begin{proof}
For $\lambda = 0$, $\nabla_{t_1}t_1+\nabla_{t_2}t_2=\nabla_{f_*{\frac{\partial}{\partial z_i}}}f_*{\frac{\partial}{\partial\overline{z_i}}} = 0$, since $f$ is holomorphic.
Similarly, the result holds for $\lambda = 1$.
\end{proof}

Let $$f_*\frac{\partial}{\partial z_i} = \sum\limits_{s=1}^{m}(a_s(e_{2s-1}-\sqrt{-1}e_{2s})+b_s(e_{2s}+\sqrt{-1}e_{2s-1}));$$ $$g_*\frac{\partial}{\partial z_i} = \sum\limits_{s=1}^{m}c_s(e_{2s-1}-\sqrt{-1}e_{2s})+d_s(e_{2s}+\sqrt{-1}e_{2s-1})).$$ Also let $$\frac{\partial}{\partial t_1}|_{t_1=0} = \sum\limits_{s=1}^{m}(u_{s}e_{2s-1}+v_{s}e_{2s}); \frac{\partial}{\partial t_2}|_{t_2=0} = \sum\limits_{s=1}^{m}(\tilde u_{s}e_{2s-1}+\tilde v_{s}e_{2s}).$$ Note that $$v_1(0) = b_1, \tilde v_1(0) = -a_1, v_1(1) = d_1, \tilde v_1(1) = -c_1.$$ 
By (\ref{3}) and Lemma \ref{lm-0},
\begin{equation}
\begin{aligned}
\frac{d^2L_1(t_1)}{dt_1^2}+\frac{d^2L_2(t_2)}{dt^2}&\geq \frac{\int\limits_{0}^{1}|v_1'(\lambda)|^2+|\tilde v_1'(\lambda)|^2+\sum\limits_{s=2}^{m}(|u_s'(\lambda)|^2+|v_s'(\lambda)|^2+|\tilde u_s'(\lambda)|^2+|\tilde v_s'(\lambda)|^2)d\lambda}{L(0)}\\&\geq  \frac{\int\limits_{0}^{1}(|v_1'(\lambda)|^2+|\tilde v_1'(\lambda)|^2)d\lambda}{L(0)}\\&\geq \frac{(b_1-d_1)^2+(a_1-c_1)^2}{L(0)}.
\end{aligned}
\end{equation}

  In the last step, we used the Cauchy-Schwarz inequality. By (\ref{2}), $$\frac{dL_1(t)}{dt} = c_1-a_1, \frac{dL_2(t)}{dt} = b_1-d_1.$$ The lemma follows if we plug these (\ref{0}).\end{proof}
  
It is clear $\iota$ is a continuous map.  
  Let $f, g\in H_{k}$. Set $M(r) = \max d(f(x), g(x))$ for $x\in B(o, r)$. Then there exists a constant $C_1>0$ such that $M(r)\leq C_1r^k$ for $r$ sufficiently large. Lemma \ref{lem1} and previous arguments imply that 
  \begin{lemma}\label{lm34}
  $\frac{M(r)}{r^k}$ is nonincreasing.
  \end{lemma}
  If $\iota(f) = \iota(g)$, there exists a constant $C$ such that $M(r) \leq Cr^{1+[k]}$ for all small $r$.   Letting $r\to 0$, we find $f = g$. This proves that $\iota: H_k\to J^{[k]}_mN$ is injective.    

To show that $\iota^{-1}: \iota(H_k)\to H_k$ is a continuous map, we consider a sequence $s_i \to s$ in $\iota(H_k)$. It suffices to show $\iota^{-1}(s_i)\to \iota^{-1}s$. Since $\iota$ is injective, we only need to show that a subsequence of $\iota^{-1}(s_i)$ converges. This is a direct consequence of proposition \ref{prop-9}. Therefore $\iota$ is an embedding.

Now assume $k$ is a positive integer and $\iota$ is a homeomorphism. Let $f, g\in H_k$. By lemma \ref{lm34}, \begin{equation}\label{6}
d(f(x), g(x)) \leq Ar^k \end{equation} for all $r\geq 1$. Here $A = \max\limits_{r\leq 1}d(f(x), g(x))$. Consider a sequence of holomorphic maps $f_i\in H_{k}$ converging to $f$ uniformly in each compact set of $M$, then $$A_i = \max\limits_{r\leq 1}d(f_i(x), f(x))\to 0$$ as $i\to\infty$. After taking a subsequence, the limit of  the ``difference" between $f_i$ and $f$, with a suitable normalization, will be a holomorphic section $\Gamma\in f^*{T^{1,0}N}$. By (\ref{6}), $\Gamma\in \mathcal{O}_k(M, f^*{T^{1, 0}N})$. Since $N$ has nonpositive sectional curvature, $f^*{T^{1,0}N}$ will be nonpositive in the sense of Griffith.  By theorem \ref{thm5}, 
\begin{equation}\label{47}
dim(\mathcal{O}_k(M, f^*{T^{1, 0}N}))\leq dim(\mathcal{O}_k(\mathbb{C}^m, E)) = dim(J^{k}_mN).
\end{equation}
Here $E$ is the trivial flat bundle on $\mathbb{C}^m$ with rank $n$.

Let $f_0$ be a constant map, say $f_0(M) = q\in N$. Since $\iota$ is a homeomorphism, by looking at sequences of functions $f_i\to f_0$, we can show that (\ref{47}) is an equality (The $k$-jet image of $\Gamma$ at $o$ can be identified with the tangent space of $J^{k}_mN$ at $\iota(f_0)$). The rigidity of theorem \ref{thm5} says $M$ is isometric and biholomorphic to $\mathbb{C}^m$.

Next we show $N$ is flat. Consider $q\in N$, $X\in T_qN$.  Choose a holomorphic coordinate $V(w_1, w_2, \cdot\cdot\cdot, w_n)$ around $q$ such that $\frac{\partial}{\partial w_1} = X - \sqrt{-1}JX$. Let $o$ be the origin of $M = \mathbb{C}^m$. By the assumption that $\iota$ is a homeomorphism, we can find $f\in H_k$ so that $$f(z_1, z_2, \cdot\cdot\cdot, z_m) = (z_1^k+O(r^{k+1}), O(r^{k+1}), \cdot\cdot\cdot, O(r^{k+1})) \in V.$$ Recall $M(f, r) = \max d(f(x), f(o))$ for $x\in B(o, r)$.  Arguments similar to corollary \ref{cor1} and corollary \ref{cor2} imply that $\frac{M(f, r)}{r^k}$ is constant.  By checking the equality case in lemma \ref{lm-7}, we find $R^N_{XJXJXX} = 0$. This means $N$ is flat. 
Since $N$ is simply connected, $N$ is isometric and biholomorphic to $\mathbb{C}^n$. 
The proof of theorem \ref{thm6} is complete.

\end{proof}

\section{\bf{Hadamard Three Circle Theorem, General Case}}
We introduce the three circle theorem when the holomorphic sectional curvature has a lower bound.
\begin{theorem}\label{thm8}
Let $M^n$ be a complete noncompact K\"ahler manifold, $p\in M$. Let $r(x) = dist(p, x)$, $e_1 = \frac{1}{\sqrt{2}}(\nabla r - \sqrt{-1}J\nabla r)$.  Suppose there exists a continuous function  $g(r)$ such that the holomorphic sectional curvature satisfies $R_{1\overline{1}1\overline{1}}\geq g(r)$ for all $r\geq 0$. 
Let $u(r)\in C^1(0, \infty)$ satisfy $$u' + 2u^2 + \frac{g}{2}\geq 0$$ and $\lim\limits_{r\to 0^+}2u(r)r=1$.
Suppose $h(r)\in C^1(0, \infty)$ satisfies \begin{itemize} \item $\lim\limits_{r\to 0^+}\frac{e^{h(r)}}{r} = 1$; \item $h'(r) > 0$; \item
$\frac{1}{2}h''+h'u(r)\leq 0$. \end{itemize}Then for any $f\in \mathcal{O}(M)$, $\log M_f(r)$ is convex in terms of $h(r)$. In particular, if $f\in  \mathcal{O}(M)$ vanishes at $p$ with order $d$, $\frac{M_f(r)}{e^{dh(r)}}$ is nondecreasing.
\end{theorem}
\begin{remark}
When $g(r) = 0$, we can take $u(r) = \frac{1}{2r}, h(r) = \log r$. This recovers the three circle theorem when the holomorphic sectional curvature is nonnegative.
It is also possible to generalize theorem \ref{thm8} to Hermitian holomorphic vector bundle cases. \end{remark}
\begin{remark}
Theorem \ref{thm8} is sharp in general. One can consider the unitary symmetric metric on $\mathbb{C}^n$ or the unit ball in $\mathbb{C}^n$ with the prescribed holomorphic sectional curvature in the radial direction.  Then any homogeneous polynomial holomorphic function satisfies the equality.
\end{remark}

\begin{proof}
The argument is similar to theorem \ref{thm-10}.
We only prove the monotonicity formula here.  According to the assumption, $r_{1\overline 1}\leq u(r)$.
 Let $0 < r_1 \leq r_2$. For small $\epsilon>0$, define $$F_{\epsilon}(x) = \log |f(x)| - \log M_f(r_2)-(d-\epsilon)(h(r)-h(r_2)).$$ Then $F_{\epsilon}(x) \leq 0$ for $r(x) = r_2$ or $r<\delta$ where $\delta<r_1$ is very small. Also define the perturbation function $$h_{\epsilon_1} = \log(e^h-\epsilon_1).$$ Then $$(h_{\epsilon_1})_{1\overline1}<0.$$ Write $$F_{\epsilon, \epsilon_1}(x) = \log |f(x)| - \log M_f(r_2)-(d-\epsilon)(h_{\epsilon_1}(r)-h_{\epsilon_1}(r_2)).$$ Here $\epsilon_1<< \delta$ so that $F_{\epsilon, \epsilon_1}(x)\leq 0$ for $r=\delta$.  We apply the maximum principle for $F_{\epsilon, \epsilon_1}$ on $B(p, r_2)-B(p, \delta)$. If the maximum point $q_{\epsilon_1}$ is in the interior part and not on the cut locus of $p$, $(F_{\epsilon, \epsilon_1})_{1\overline1}(q_{\epsilon_1})\leq0$. By direct computation, $(F_{\epsilon, \epsilon_1})_{1\overline1}>0$. This is a contradiction. If $q_{\epsilon_1}$ is on the cut locus of $p$, we employ Calabi's trick. Namely, define $\hat r=dist(p_1, x)$ where $p_1$ is the point on the minimal geodesic connecting $p$ and $q_{\epsilon_1}$ such that $dist(p, p_1)=\epsilon_2$.  Define $$F_{\epsilon, \epsilon_1, \epsilon_2}(x) = \log |f(x)| - \log M_f(r_2)-(d-\epsilon)(h_{\epsilon_1}(\hat r+\epsilon_2)-h_{\epsilon_1}(r_2))\leq F_{\epsilon, \epsilon}(x)$$ on $B(p, r_2)-B(p, \delta)$. Here we used the triangle inequality and that $h' > 0$. Moreover, $F_{\epsilon, \epsilon_1, \epsilon_2}(q_{\epsilon_1})= F_{\epsilon, \epsilon_1}(q_{\epsilon_1})$. Therefore $q_{\epsilon_1}$ is the maximum of $F_{\epsilon, \epsilon_1,\epsilon_2}$. Thus $$(F_{\epsilon, \epsilon_1,\epsilon_2})_{1\overline1}\leq0.$$ By continuity of the holomorphic sectional curvature, for fixed $\epsilon_1>0$, if $\epsilon_2$ is small enough, $$(F_{\epsilon, \epsilon_1,\epsilon_2})_{1\overline1}>0$$ at $q_{\epsilon_1}$. This is a contradiction.
  Then $$F_{\epsilon, \epsilon_1, \epsilon_2}(x)\leq 0$$ for $\delta\leq r\leq r_2$. Letting $\epsilon_2\to 0, \epsilon_1\to 0, \epsilon\to 0$ and putting $r(x) = r_1$, we find $$\log M_f(r_1) -\log M_f(r_2) - d(h(r_1)-h(r_2))\leq 0.$$ This is equivalent to $$\frac{M_f(r_1)}{e^{h(r_1)d}}\leq \frac{M_f(r_2)}{e^{h(r_2)d}}.$$
 \end{proof}

Below are two corollaries of theorem \ref{thm8}:
 \begin{theorem}
 Let $M^n$ be a complete noncompact K\"ahler manifold such that the holomorphic sectional curvature is no less than $-1$. Let $p\in M$ and $r(x) = dist(p, x)$. Then for $f\in\mathcal{O}(M)$,  $M_f(r)$ is convex in terms of $\log \frac{e^r-1}{e^r+1}$. In particular, $\frac{M_f(r)}{(\frac{e^r-1}{e^r+1})^d}$ is nondecreasing where $d$ is the vanishing order of $f$ at $p$.
 \end{theorem}
 \begin{theorem}
 Let $M$ be a complete (compact) K\"ahler manifold with holomorphic sectional curvature no less than $1$, $p\in M$. Let $f$ be a holomorphic function defined on $B(p, r_0) \subset M$. Then $M_f(r)$ is convex in terms of $\log(\tan\frac{r}{2})$. In particular, if $f$ vanishes at $p$ with order $d$, $\frac{M_f(r)}{(\tan\frac{r}{2})^d}$ is nondecreasing. 
 \end{theorem}

\section{\bf{Complete K\"ahler manifolds with holomorphic sectional curvature asymptotically nonnegative }}
In this section we apply theorem \ref{thm8} to K\"ahler manifolds with holomorphic sectional curvature asymptotically nonnegative.
\begin{theorem}\label{thm7}
Let $M^n$ be a complete noncompact K\"ahler manifold, $p\in M$, $r(x)=dist(p, x)$. Suppose there exist constants $\epsilon, A >0$ such that for any $e_i\in T_x^{1,0}M$ with unit length, $R_{i\overline{i}i\overline{i}} \geq -\frac{A}{(r+1)^{2+\epsilon}}$. Then there exists $C = C(A, \epsilon)>0$ such that for any positive integer $d$, $$dim(\mathcal{O}_d(M))\leq Cd^{n}.$$  Thus the power of $d$ is sharp comparing with the complex Euclidean space. If $d\leq e^{\frac{-3A}{\epsilon}}$, $dim(\mathcal{O}_d(M)) = 1$. Finally, if $d$ is a positive integer satisfying $\frac{A}{\epsilon}\leq \frac{1}{4d}$, then $$dim(\mathcal{O}_d(M))\leq dim(\mathcal{O}_d(\mathbb{C}^n)).$$  \end{theorem}

\begin{remark}
Theorem \ref{thm7} is not true for $\epsilon \leq 0$. In this case, it is possible that the dimension of bounded holomorphic functions is infinite. For example,  on page $226$ of \cite{[SY]}, it was remarked that on the unit disk, the metric $ds^2 = \frac{1}{(1-|z|^2)^m}dz\otimes d\overline{z}(m\geq 3)$ has quadratic holomorphic sectional curvature decay.\end{remark}
\begin{remark}
This theorem says the sharp dimension estimate is ``stable" under the variation of curvature (the topology might change). Also note that so far, there is no Riemannian analogue of    
theorem \ref{thm7}. There, some extra topological conditions are required. For example, see \cite{[W]}\cite{[T]}.
\end{remark}

\begin{proof}
We may assume $\epsilon<\frac{1}{2}$. 
Take $u(r) = \frac{1}{2r} + \frac{A}{(1+r)^{1+\epsilon}}$, then $$2u^2 + u' - \frac{1}{2}\frac{A}{(1+r)^{2+\epsilon}}\leq 0.$$
Let $h(r)$ be the solution to the equation $\frac{1}{2}h''+h'u= 0$ with $\lim\limits_{r\to 0^+}\frac{e^{h(r)}}{r} = 1$. Then $$h'(r) = \frac{e^{\frac{2A}{\epsilon(r+1)^{\epsilon}}}}{r}e^{-\frac{2A}{\epsilon}}.$$ Therefore, \begin{equation}\begin{aligned}
h(r) = \int\limits^r_{1}\frac{e^{\frac{2A}{\epsilon(t+1)^{\epsilon}}}}{t}e^{-\frac{2A}{\epsilon}}dt+C\geq e^{-\frac{2A}{\epsilon}}\ln r+C\end{aligned}\end{equation} for $r\geq 1$. By theorem \ref{thm8}, if $f\in \mathcal{O}(M)$ and $f$ vanishes at $p$ with order $d$, $\frac{M_f(r)}{e^{h(r)d}}$ is nondecreasing. Thus $$M_f(R) \geq e^{h(R)d}\lim\limits_{r\to 0}\frac{M_f(r)}{e^{h(r)d}} \geq C_1e^{Cd}R^{de^{-\frac{2A}{\epsilon}}}.$$ Here $C_1 = \lim\limits_{r\to 0}\frac{M_f(r)}{e^{h(r)d}}>0$. By a linear algebra argument, $$dim(\mathcal{O}_d(M)) \leq dim(\mathcal{O}_{de^{\frac{2A}{\epsilon}}}(\mathbb{C}^n))=C(A, \epsilon)d^n.$$
If $d\leq e^{\frac{-3A}{\epsilon}}$,  $$dim(\mathcal{O}_d(M)) \leq dim(\mathcal{O}_{de^{\frac{2A}{\epsilon}}}(\mathbb{C}^n))\leq dim(\mathcal{O}_{e^{\frac{-A}{\epsilon}}}(\mathbb{C}^n)) = 1.$$
Finally, if  $\frac{A}{\epsilon}\leq \frac{1}{4d}\leq \frac{1}{4}$, $e^{\frac{2A}{\epsilon}}d< d+1$. Therefore $$dim(\mathcal{O}_d(M)) \leq dim(\mathcal{O}_{de^{\frac{2A}{\epsilon}}}(\mathbb{C}^n))= dim(\mathcal{O}_d(\mathbb{C}^n)).$$\end{proof}
A similar argument yields
\begin{cor}\label{cor4}
Let $M^n$ be a complete noncompact K\"ahler manifold such that the holomorphic sectional curvature is nonnegative outside a compact set $K$. Let $-a$ $(a > 0)$ be the lower bound of the holomorphic sectional curvature on $M$. Define $\lambda = a(diam(K))^2$. Then there exists a positive constant $C$ depending only on $\lambda$ such that for any positive integer $d$, $$dim(\mathcal{O}_d(M))\leq Cd^{n}.$$ Thus the power is sharp.  If $\lambda \leq \frac{1}{2d}$ for some positive integer $d$, we have the sharp dimension estimate $$dim(\mathcal{O}_d(M))\leq dim(\mathcal{O}_d(\mathbb{C}^n)).$$ Finally, there exists a small number $\epsilon = \epsilon(\lambda)>0$ such that $$dim(\mathcal{O}_{\epsilon}(M)) = 1.$$
\end{cor}
\begin{remark}
The inequality $dim(\mathcal{O}_d(M))\leq dim(\mathcal{O}_d(\mathbb{C}^n))$ does not hold if we only assume $M$ has nonnegative holomorphic sectional curvature outside a compact set. One can easily construct an example in complex one dimensional case (the rotationally symmetric case on $\mathbb{C}$).
\end{remark}
\begin{cor}\label{cor5}
Theorem \ref{thm4} is valid under the assumption of theorem \ref{thm7}. 
\end{cor}
\begin{proof}
We just replace $\log r$ by $h(r)$ in the proof of theorem \ref{thm4}. Then everything works as the same.
\end{proof}

\begin{remark}
We can also generalize theorem \ref{thm5} and theorem \ref{thm6} as in corollary \ref{cor5}. 
\end{remark}

\begin{cor}\label{cor-1}
Under the same assumption as in theorem \ref{thm7}, any holomorphic map from $M$ to $N$ must be a constant map. Here $N$ is a complete simply connected K\"ahler manifold with nonpositive sectional curvature and that the sectional curvature satisfies $sec \leq -\frac{K}{d^2}$ for all large $d$. Here $d$ is the distance function to a point $q$ on $N$, $K$ is a positive constant.
\end{cor}
\begin{proof}
The proof will be sketchy. Let $f$ be a holomorphic map from $M$ to $N$ and $q'\in f(M)$. From the comparison theorem, it is easy to construct a nonconstant bounded plurisubharmonic function $g (x)= \rho(d'(x))$ on $N$. Here $d'(x) = d(q', x)$. We obtain a bounded plurisubharmonic function $f^*g$ on $M$. By corollary \ref{cor5}, $f^*g$ is a constant. Since $q'\in f(M)$, $f^*g\equiv f(q')$. 
As $q'$ is the unique minimum point of $g$, $f(M) = q'$.

\end{proof}

\section{\bf{Miscellaneous Results }}

\begin{theorem}
Let $M^n$ be a complete noncompact K\"ahler manifold with nonnegative holomorphic sectional curvature. Suppose the holomorphic sectional curvature is positive at one point, then there exists $\epsilon>0$ depending only on $M$ such that for any integer $d\geq 1$, $dim(\mathcal{O}_d(M))\leq dim(\mathcal{O}_{(1-\epsilon)d}(\mathbb{C}^n))$.
\end{theorem}
\begin{remark}
Similar results were proved by Chen, Fu, Le and Zhu \cite{[CFLZ]}.
\end{remark}
\begin{proof}
We can estimate $h(r)$ in theorem \ref{thm8}. For large $r$, it turns out that $h(r)\geq (1+\delta)\log r+C$ (here $\delta>0, C$ are constants depending only on $M$). The details will be omitted. The theorem follows from the standard linear algebra argument as before.
\end{proof}
Below we show that if the holomorphic sectional curvature is positive near infinity and does not decay too fast,  the dimension of holomorphic functions with exponential growth is finite.
\begin{definition}
Let $M$ be a complete noncompact K\"ahler manifold and $p\in M$. Let $r(x) = dist(x, p)$. For any $A>0, d\geq 1$, define $E_{A, d}(M)=\{f\in \mathcal{O}(M)|\overline{\lim\limits_{r\to\infty}}\frac{|f(x)|}{e^{dr^A}}<\infty\}$.
For simplicity, we denote $E_{1, d}(M)$ by $E_d(M)$.
\end{definition}

\begin{theorem}\label{thm9}
Let $M^n$ be a complete noncompact K\"ahler manifold and $p\in M$. Let $r(x) = dist(x, p)$. Suppose the holomorphic sectional curvature satisfies $R_{1\overline{1}1\overline{1}} \geq \frac{C}{r^2}$ for all sufficiently large $r$. Here $e_1 = \frac{1}{\sqrt2}(\nabla r -\sqrt{-1}J\nabla r)$; $C$ is a constant satisfying $0<C<\frac{1}{4}$. Let $a>\frac{1}{4}$ satisfy $2a^2-a+\frac{C}{2} = 0$. Define $A = 1-2a$. Then for any $d\geq 1,$ \begin{equation}\label{24} dim(E_{A,d}(M))\leq C_1d^n,\end{equation} here $C_1$ is a constant depending only on $M$. Moreover, the power in (\ref{24}) is sharp. Finally, $\mathcal{O}_d(M) \equiv constant$ for any $d\geq 0$. 
\end{theorem}
\begin{proof}
Recall 
\begin{equation}\label{25}
2r_{1\overline{1}}^2 + \frac{\partial r_{1\overline{1}}}{\partial r} + \frac{1}{2}R_{1\overline{1}1\overline{1}}\leq 0.
\end{equation}
\begin{claim}\label{claim2}
Let $a>\frac{1}{4}, b<\frac{1}{4}$ be the positive solutions to $2x^2-x+\frac{C}{2} = 0$ for $0<C<\frac{1}{4}$. Then there exists a constant $B>0$ depending only on $M$ such that  $$r_{1\overline1}\leq \frac{aBr^k-b}{r(Br^k-1)}$$ for large $r$. Here $k = 2a-2b > 0$.
\end{claim}
\begin{proof}
It is easy to see $a<\frac{1}{2}$. Let $g(r) = rr_{1\overline1}$. Then we can plug in (\ref{25}) to find \begin{equation}\label{26} rg'-g+2g^2+\frac{C}{2}=rg'+2(g-a)(g-b)\leq 0\end{equation} for $r\geq r_0$ (here $r_0$ is a large number).  We may assume $g(r_0)\geq a$, otherwise $g(r)\leq a$ for all $r\geq r_0$ and the conclusion is obvious. Solving inequality (\ref{26}), we find \begin{equation}\label{27}\frac{1}{a-b}\ln\frac{g-a}{g-b}\leq-2\ln r +B_1.\end{equation} Here $B_1$ is a constant depending only on $g(r_0)$, $r_0$, $a$, $b$. Let $B = e^{-B_1}$, then a simplification of (\ref{27}) gives the claim.
\end{proof}

Let $h(r)$ be the solution to $\frac{1}{2}h''+h'\frac{aBr^k-b}{r(Br^k-1)}= 0$ with $\lim\limits_{r\to 0}\frac{e^{h(r)}}{r}=1$. It is not hard to see that for $r$ very large, \begin{equation}\label{28}
h(r)\geq c_1r^{1-2a}+c_2. \end{equation}Here $c_1, c_2$ are constants and $c_1>0$.
By theorem \ref{thm8}, given any $f\in \mathcal{O}(M)$, if the vanishing order at $p$ is no less than $d$, $\frac{M_f(x)}{e^{dh(r)}}$ is nondecreasing.  (\ref{28}) and a linear algebra argument give the dimension estimate. If $f\in\mathcal{O}_d(M)$ for some $d\geq 0$, then for any integer $k\geq 1$, $f^k\in E_{A, 1}(M)$. If $f$ is not a constant, $\{f^k\}$ will be linearly independent.
This contradicts that $dim(E_{A, 1}(M))$ is finite.

Finally, one can easily construct a unitary symmetric metric in $\mathbb{C}^n$ to verify the sharpness of the power in (\ref{24}).

\end{proof}
\begin{remark}
For $C = \frac{1}{4}$, we can also get sharp estimates by solving differential equations. However, the expression is more messy. If $C>\frac{1}{4}$, $M$ is automatically compact. \end{remark}
In the next theorem the sharp model is the generalized cigar soliton. This is the unitary symmetric K\"ahler metric on $\mathbb{C}^n$ defined by $\partial\overline\partial p(r)$ ($r$ is the Euclidean distance on $\mathbb{C}^n$). Here $p(r)$ is a function such that $\partial\overline\partial p(r)$ defines the cigar soliton on $\mathbb{C}$.  For details of the cigar soliton, see \cite{[H2]} and \cite{[CK]}.

 \begin{theorem}
 Let $M$ be a complete noncompact K\"ahler manifold and $p\in M$. Let $r(x) = dist(x, p)$. Suppose the holomorphic sectional curvature satisfies $R_{1\overline{1}1\overline{1}} (x)\geq \frac{8}{(e^r+e^{-r})^2}$ where $e_1 = \frac{1}{\sqrt2}(\nabla r - \sqrt{-1}J\nabla r)$. We have the sharp dimension estimate $$dim(E_d(M))\leq dim(E_d(N))=dim(\mathcal{O}_d(\mathbb{C}^n))$$ where $N$ is the generalized cigar soliton.
 \end{theorem}
 \begin{proof}
 Define $h(r) =\log\frac{e^r-e^{-r}}{2}$. We show that $(h(r))_{1\overline{1}}\leq 0$. It suffices to verify $r_{1\overline{1}}\leq \frac{2}{e^{2r}-e^{-2r}}$. This follows by solving (\ref{25}). 
 Let $f\in \mathcal{O}(M)$ which vanishes at $p$ with order $k$. By theorem \ref{thm8}, $\frac{M_f(r)}{e^{kh(r)}}$ is nondecreasing. This implies that if $f\in E_d(M)$, the vanishing order of $f$ at $p$ is no greater than $d$. By a linear algebra argument, the sharp dimension inequality holds.
 \end{proof} 
 
\section{\bf{Examples}}
In this section we give some examples showing that certain K\"ahler manifolds admit complete K\"ahler metric with positive holomorphic sectional curvature, yet do not admit complete K\"ahler metric with nonnegative Ricci curvature. Recall the Hirzebruch surface $\mathbb{F}_n$ is $\mathbb{P}(\mathcal{O}_{\mathbb{P}^1}(n)\oplus \mathcal{O}_{\mathbb{P}^1})$, $n\geq 0$.  \begin{prop}[Hitchin]
For any natural number $n$ and positive integer $m$, complex manifold $\mathbb{F}_n\times \mathbb{C}^m$ admits complete K\"ahler metric with positive holomorphic sectional curvature. For infinitely many $n$, those manifolds do not admit complete K\"ahler metric with nonnegative Ricci curvature.
\end{prop}

\begin{proof}
Recall in \cite{[H]}, Hitchin showed that for any integer $n\geq 0$, $\mathbb{F}_n$ admits a K\"ahler metric with positive holomorphic sectional curvature. 
Since $\mathbb{C}^m$ admits a complete K\"ahler metric with positive holomorphic bisectional curvature, the product metric $\mathbb{F}_n\times \mathbb{C}^m$ has positive holomorphic sectional curvature.

\bigskip

 Now assume $M = \mathbb{F}_n\times \mathbb{C}^m$ admits a K\"ahler metric with nonnegative Ricci curvature. We have $$TM|_{\mathbb{F}_n} = T\mathbb{F}_n\oplus N$$ where $N$ is the trivial normal bundle over $\mathbb{F}_n$. Therefore $c_1(T\mathbb{F}_n)\geq 0$. According to \cite{[GH]}, for $\mathbb{F}_n$, the canonical divisor $$K = -2E_0+(n-2)C$$ where $E_0$ is the zero-section of $\mathbb{F}_n$ and $C$ is the fibre. Consider the section $(\sigma, 0)\in \mathcal{O}_{\mathbb{P}^1}(n)\oplus \mathcal{O}_{\mathbb{P}^1}$,  where $\sigma$ is any section on $\mathcal{O}_{\mathbb{P}^1}(n)$. Away from the zeros of $\sigma$, $(\sigma, 0)$ gives a curve in $\mathbb{F}_n$. Let $E_\infty$ be the closure of this curve. We have $E_0\cdot E_0 = n$; $E_0\cdot C = 1$; $E_\infty\sim E_0-nC$. Since $K\cdot E_\infty\leq 0$, $n\leq 2$. This concludes the proof of the proposition.

\end{proof}

\end{document}